\documentclass{amsart}

\usepackage{enumerate}
\usepackage{amsmath}%
\usepackage{amsfonts}%
\usepackage{amssymb}%
\usepackage{graphicx}
\usepackage{mathrsfs}
\usepackage{hyperref}
\usepackage[margin=1in]{geometry}

\newtheorem{theorem}{Theorem}
\theoremstyle{plain}

\newtheorem{claim}[theorem]{Claim}

\newtheorem{corollary}[theorem]{Corollary}

\newtheorem{fact}[theorem]{Fact}
\newtheorem{lemma}[theorem]{Lemma}

\newtheorem{proposition}[theorem]{Proposition}

\numberwithin{equation}{section}
\numberwithin{theorem}{section}
\numberwithin{case}{section}

\numberwithin{subcase}{case}

\def\A{\mathcal{A}}

\def\F{\mathcal{F}}

\def\V{\mathcal{V}}
\def \a{\alpha}
\def \b{\beta}
\def \e{\epsilon}
\def \r{\gamma}

\def \bfi{\mathbf{i}}
\def \bfu{\mathbf{u}}
\def \bfv{\mathbf{v}}
\def \bfw{\mathbf{w}}
\def\cP{\mathcal{P}}

\begin{document}
\title{Matchings in $k$-partite $k$-uniform Hypergraphs}
\thanks{
The first author is supported by FAPESP (Proc. 2013/03447-6, 2014/18641-5, 2015/07869-8).
The third author is partially supported by NSF grants DMS-1400073 and DMS-1700622.}
\author{Jie Han}
\address{Instituto de Matem\'{a}tica e Estat\'{\i}stica, Universidade de S\~{a}o Paulo, Rua do Mat\~{a}o 1010, 05508-090, S\~{a}o Paulo, Brazil}
\email[Jie Han]{jhan@ime.usp.br}
\author{Chuanyun Zang}
\author{Yi Zhao}
\address
{Department of Mathematics and Statistics, Georgia State University, Atlanta, GA 30303}
\email[Chuanyun Zang]{chuanyun.zang@gmail.com}
\email[Yi Zhao]{yzhao6@gsu.edu}

\date{\today}
\maketitle

\begin{abstract}
For $k\ge 3$ and $\e>0$, let $H$ be a $k$-partite $k$-graph with parts $V_1,\dots, V_k$ each of size $n$, where $n$ is sufficiently large.
Assume that for each $i\in [k]$, every $(k-1)$-set in $\prod_{j\in [k]\setminus \{i\}} V_i$ lies in at least $a_i$ edges, and $a_1\ge a_2\ge \cdots \ge a_k$. We show that if $a_1, a_2\ge \e n$, then $H$ contains a matching of size $\min\{n-1, \sum_{i\in [k]}a_i\}$.
In particular, $H$ contains a matching of size $n-1$ if each crossing $(k-1)$-set lies in at least $\lceil n/k \rceil$ edges, or each crossing $(k-1)$-set lies in at least $\lfloor n/k \rfloor$ edges and $n\equiv 1\bmod k$. This special case answers a question of R\"odl and Ruci\'nski and was independently obtained by Lu, Wang, and Yu. 

The proof of Lu, Wang, and Yu closely follows the approach of Han [Combin. Probab. Comput. 24 (2015), 723--732] by using the absorbing method and considering an extremal case. In contrast, our result is more general and its proof is thus more involved: it uses a more complex absorbing method and deals with two extremal cases.
\end{abstract}

\section{Introduction}

A \emph{$k$-uniform hypergraph} (in short, \emph{$k$-graph}) consists of a vertex set $V$ and an edge set $E\subseteq \binom{V}{k}$, that is, every edge is a $k$-element subset of $V$.
A $k$-graph $H$ is \emph{$k$-partite} if $V(H)$ can be partitioned into $k$ parts $V_1, \dots, V_k$ such that every edge consists of exactly one vertex from each class, in other words, $E(H)\subseteq V_1\times \cdots \times V_k$.
A \emph{matching} in $H$ is a collection of vertex-disjoint edges of $H$. A matching covering all vertices of $H$ is called \emph{perfect}.

Given a $k$-graph $H$ and a set $S$ of $d$ vertices in $V(H)$, where $1\le d\le k-1$, a \emph{neighbor} of $S$ is a $(k-d)$-set $T\subseteq V(H)\setminus S$ such that $S\cup T\in E(H)$. Denote by $N_H(S)$ the set of the neighbors of $S$, and define 
the \emph{degree} of $S$ to be $\deg_H(S) = |N_H(S)|$. 
We omit the subscript $H$ if it is clear from the context. The \emph{minimum $d$-degree $\delta _{d}(H)$} of $H$ is the minimum of $\deg_H(S)$ over all $d$-subsets $S$ of $V(H)$.
The minimum $(k-1)$-degree is also called the \emph{minimum codegree}.

The minimum $d$-degree thresholds that force a perfect matching in $k$-graphs have been studied intensively, see \cite{AFHRRS,CzKa,HPS,Han16_mat,Khan1,Khan2,KOT,MaRu,Pik, RRS06mat, RRS09, TrZh12,TrZh13,TrZh15} and surveys \cite{RR, zsurvey}.
In particular, R\"odl, Ruci\'nski and Szemer\'edi \cite{RRS09} determined the minimum codegree threshold that guarantees a perfect matching in an $n$-vertex $k$-graph for large $n$ and all $k\ge 3$. The threshold is $n/2-k+C$, where $C\in\{3/2, 2, 5/2, 3\}$ depending on the values of $n$ and $k$.
In contrast, the minimum codegree threshold for a matching of size $\lceil n/k\rceil - 1$ is much smaller. 
R\"odl, Ruci\'nski and Szemer\'edi \cite{RRS09} showed that 
every $k$-graph $H$ on $n$ vertices satisfying $\delta_{k-1}(H)\ge n/k+O(\log n)$ contains a matching of size $\lceil n/k\rceil - 1$.
Han \cite{Han14_mat} improved this by reducing the assumption to $\delta_{k-1}(H)\ge \lceil n/k\rceil - 1$, which is  best possible.

In this paper we are interested in the corresponding thresholds in $k$-partite $k$-graphs.
Suppose $H$ is a $k$-partite $k$-graph with parts $V_1, \dots, V_k$. A subset $S\subset V(H)$ is called \emph{crossing} if $|S\cap V_i|\le 1$ for all $i$.
For any $I\subseteq [k]$, let $\delta_I(H)$ be the minimum of $\deg_{H}(S)$ taken over all crossing $|I|$-vertex sets $S$ in $\prod_{i\in I}V_i$.
Then the \emph{partite minimum $d$-degree} $\delta'_d(H)$ is defined as the minimum of $\delta_I(H)$ over all $d$-element sets $I\subseteq [k]$.

Let $H$ be a $k$-partite $k$-graph with $n$ vertices in each part.
For $k\ge 3$, K\"uhn and Osthus \cite{KO06mat} proved that if $\delta_{k-1}'(H) \ge n/2+\sqrt{2n\log n}$ then $H$ has a perfect matching. Later Aharoni, Georgakopoulos and Spr\"ussel~\cite{AGS} improved this result by requiring only two partite minimum codegrees.
They showed that $H$ contains a perfect matching if $\delta_{[k]\setminus {\{1\}}}(H)> n/2$ and $\delta_{[k]\setminus {\{2\}}}(H)\ge n/2$, and consequently, if $\delta'_{k-1}(H)>n/2$ then $H$ has a perfect matching.

Similarly to the non-partite case, when targeting almost perfect matchings, the minimum degree threshold also drops significantly. K\"uhn and Osthus in \cite{KO06mat} proved that $\delta'_{k-1}(H)\ge \lceil n/k \rceil$ guarantees a matching of size $n-(k-2)$.
R\"odl and Ruci\'nski \cite[Problem 3.14]{RR} asked whether $\delta'_{k-1}(H)\ge \lceil n/k \rceil$ guarantees a matching in $H$ of size $n-1$.
In this paper, we answer this question in the affirmative and show that the threshold can be actually weakened to $\lfloor n/k \rfloor$ if $n\equiv 1\pmod k$.
In fact, our result is much more general -- it only requires that the sum of the partite minimum codegrees is large and at least two partite codegrees are not small.

\begin{theorem}[Main Result]\label{thm:main}
For any $k\ge 3$ and $\e>0$, there exists $n_0$ such that the following holds for all $n\ge n_0$.
Let H be a $k$-partite $k$-graph with parts of size $n$ and $a_i:=\delta_{[k]\setminus\{i\}}(H)$ for all $i\in [k]$ such that $a_1\ge a_2\ge \cdots \ge a_k$ and $a_2>\e n$.
Then $H$ contains a matching of size at least $\min\{n-1, \sum_{i=1}^k a_i\}$.
\end{theorem}

Our proof, based on the absorbing method, unfortunately fails when $a_1$ is close to $n$ and all of $a_2,\dots, a_k$ are small.  It is unclear (to us) if the same assertion holds in this case.

The following corollary follows from Theorem~\ref{thm:main} immediately. It was announced at~\cite{Zang16talk} and appeared in the dissertation of the second author \cite{Zang_thesis}. The second case of Corollary~\ref{cor} resolves \cite[Problem 3.14]{RR} and was independently proven by Lu, Wang and Yu \cite{LWY}.

\begin{corollary}\label{cor}
Given $k\ge 3$, there exists $n_0$ such that the following holds for all $n\ge n_0$.
Let H be a $k$-partite $k$-graph with parts of size $n$.
Then $H$ contains a matching of size $n-1$ if one of the following holds.
\begin{itemize}
\item $n\equiv 1\bmod k$ and $\delta'_{k-1}(H)\ge \lfloor n/k \rfloor$;
\item $\delta'_{k-1}(H)\ge \lceil n/k \rceil$.
\end{itemize}
\end{corollary}

Let $\nu(H)$ be the size of a maximum matching in $H$. The following greedy algorithm, which essentially comes from \cite{KO06mat}, gives a simple proof of Theorem \ref{thm:main} when $\sum_{i=1}^k a_i \le n-k+2$ or when $a_1+a_2\ge n-1$.

\begin{fact}\label{thm:mat_num}
Let $n\ge k-2$.
Suppose H is a $k$-partite $k$-graph with parts of size at least $n$. Let $a_i:=\delta_{[k]\setminus\{i\}}(H)$ for $i\in [k]$.
Then
\[
\nu(H)\ge \min\left\{n-k+2, \,\sum_{i=1}^k a_i \right\}  \quad \text{and} \quad \nu(H)\ge \min\{n-1, \,a_1+a_2\}.
\]
\end{fact}

\begin{proof}
Assume a maximum matching $M$ of $H$ has size $|M|\le \min\{n-k+1, \sum_{i=1}^k a_i -1\}$. Since each class has at least $k-1$ vertices unmatched, we can find $k$ disjoint crossing $(k-1)$-sets $U_1, \dots, U_{k}$ such that $U_i$ contains exactly one unmatched vertex in $V_j$ for $j\ne i$. Each $U_i$ has at least $a_i$ neighbors and all of them lie entirely in $V(M)$. Since $\sum_{i=1}^k a_i>|M|$, there exist distinct indices $i\ne j$ such that $U_i$ and $U_j$ have neighbors on the same edge $e\in M$, say $v_i\in N(U_i)\cap e$ and $v_j\in N(U_j)\cap e$. Replacing $e$ by $\{v_i\}\cup U_i$ and $\{v_j\}\cup U_j$ gives a larger matching, a contradiction.
The second inequality can be proved similarly.
\end{proof}

The following construction, sometimes called a \emph{space barrier}, shows that the degree sum conditions in Theorem~\ref{thm:main} and Fact~\ref{thm:mat_num} are best possible. 
Let $H_0 = H_0(a_1,\dots, a_k)$ be a $k$-partite $k$-graph with $n$ vertices in each part $V_1,\dots, V_k$.
For each $i\in [k]$, fix a set $A_i\subseteq V_i$ of size $a_i$. Let $E(H_0)$ consist of all crossing $k$-sets $e$ such that $e\cap A_i\neq\emptyset$ for some $i\in [k]$. 
Suppose $\sum_{i=1}^k a_i \le n-1$.
Clearly both $\nu(H_0)$ and the partite degree sum of $H_0$ equal to $\sum_{i=1}^k a_i$ (so we cannot expect a matching larger than $\sum_{i=1}^k a_i$).

Given a set $V$, let $V_1\cup \cdots \cup V_k$ and  $A\cup B$ be two partitions of $V$.
For $i\in [k]$ we always write $A_i:=A\cap V_i$ and $B_i:=B\cap V_i$.
A set $S\subseteq V$ is \emph{even} (otherwise \emph{odd}) if it intersects $A$ in an even number of vertices. 
Let $E_{even}(A, B)$ (respectively, $E_{odd}(A, B)$) denote the family of all crossing $k$-subsets of $V$ that are even (respectively, odd).

To see that we cannot always expect a perfect matching when $\sum_{i=1}^k a_i\ge n$, consider the following example, sometimes called a \emph{divisibility barrier}.
Let $H_1$ be a $k$-partite $k$-graph with $n$ vertices in each of its parts $V_1,\dots, V_k$.
For $i\in [k]$, let $V_i=A_i\cup B_i$ such that $\sum_{i=1}^k|A_i|$ is odd and for each $i\in [k]$, $n/2-1\le |A_i|\le n/2+1$.
Let $E(H_1)=E_{even}(A, B)$.
So the partite degree sum of $H_1$ is at least $k(n/2-1)$.
However, $H_1$ does not contain a perfect matching because any matching in $H_1$ covers an even number of vertices in $\bigcup_{i=1}^k A_i$ but $|\bigcup_{i=1}^k A_i|$ is odd.

When proving Corollary \ref{cor} directly, the authors of \cite{LWY, Zang_thesis} closely followed the approach used by the first author \cite{Han14_mat} by separating two cases based on whether $H$ is close to $H_0$. 
In contrast, to prove Theorem~\ref{thm:main}, we have to consider \emph{three} cases separately: when $H$ is close to $H_0$, when $H$ is close to (a weaker form of) $H_1$, and when $H$ is far from both $H_0$ and $H_1$.

Now we define two extremal cases formally.
Let $H$ be a $k$-partite $k$-graph with each part of size $n$ and let $a_i:=\delta_{[k]\setminus\{i\}}(H)$ for all $i\in [k]$.
We call $H$ \emph{$\e$-S-extremal} if $\sum_{i=1}^k a_i\le (1+ \e) n$ and $V(H)$ contains an independent set $C$ such that $|C\cap V_i|\ge n-a_i-\e n$ for each $i\in [k]$.
We call $H$ \emph{$\e$-D-extremal} if there is a partition $A\cup B$ of $V(H)$ such that
\begin{enumerate}[(i)]
\item $(1/2-\e)n\le a_1, a_2, |A_1|, |A_2| \le (1/2+\e)n$, and $a_i\le \e n$ for $i\ge 3$,
\item $|E_{even}(A, B)\setminus E(H)|\le \e n^k$ or $|E_{odd}(A, B)\setminus E(H)|\le \e n^k$.
\end{enumerate}

Our proof of Theorem \ref{thm:main} consists of the following three theorems.

Throughout the paper, we write $\alpha \ll \beta \ll \gamma$ to mean that we can choose the positive constants
$\alpha, \beta, \gamma$ from right to left. More
precisely, there are increasing functions $f$ and $g$ such that, given
$\gamma$, whenever $\beta \leq f(\gamma)$ and $\alpha \leq g(\beta)$, the subsequent statement holds.
Hierarchies of other lengths are defined similarly.
Moreover, when we use variables of the reciprocal form in the hierarchy, we implicitly assume that the variables are integers.
Throughout this paper, we omit the assumption $k\ge 3$ in the hierarchies.

\begin{theorem}[Non-extremal case]\label{thm:nonext}
Let $1/n\ll \r \ll\e\ll 1/k$.
Suppose H is a $k$-partite $k$-graph with parts of size $n$ with $a_i:=\delta_{[k]\setminus\{i\}}(H)$ for $i\in [k]$ such that $(1-\e)n\ge a_1\ge a_2\ge \cdots \ge a_k$, $a_2\ge \e n$, and $\sum_{i=1}^k a_i \ge (1-\r/5)n$.
Then one of the following holds:
\begin{itemize}
\item[(i)] $H$ contains a matching of size at least $n-1$;
\item[(ii)] $H$ is $\r$-S-extremal;
\item[(iii)] $H$ is $2k^2\e$-D-extremal.
\end{itemize}
\end{theorem}

\begin{theorem}[Extremal case I]\label{thm:ext}
Let $1/n\ll \r \ll\e\ll 1/k$.
Let H be a $k$-partite $k$-graph with parts of size $n$ and $a_i:=\delta_{[k]\setminus\{i\}}(H)$ for $i\in [k]$ such that $(1-\e)n\ge a_1\ge a_2\ge \cdots \ge a_k$ and $a_2\ge \e n$.
Suppose $H$ is $\r$-S-extremal.
Then one of the following holds:
\begin{itemize}
\item[(i)] $H$ contains a matching of size at least $\min\{n-1, \sum_{i=1}^k a_i\}$;
\item[(ii)] $H$ is $3\e$-D-extremal.
\end{itemize}
\end{theorem}

\begin{theorem}[Extremal case II]\label{thm:ext2}
Let $1/n \ll\e\ll 1/k$.
Suppose $H$ is a $k$-partite $k$-graph with parts of size $n$ and $a_i:=\delta_{[k]\setminus\{i\}}(H)$ for $i\in [k]$.
If $H$ is $\e$-D-extremal, then $H$ contains a matching of size at least $\min\{n-1, \sum_{i=1}^k a_i\}$.
\end{theorem}

\begin{proof}[Proof of Theorem \ref{thm:main}]
When $\sum_{i=1}^k a_i \le n - k +2$ or $a_1\ge (1-\e)n$, Theorem \ref{thm:main} follows from Fact \ref{thm:mat_num} immediately.
When $\sum_{i=1}^k a_i \ge n - k +3$ and $a_1\le (1-\e)n$, it follows from Theorems~\ref{thm:nonext}--\ref{thm:ext2}.
\end{proof}

The rest of the paper is organized as follows. In Section 2 we introduce two absorbing lemmas that are needed for the proof of Theorem~\ref{thm:nonext}: Lemma~\ref{lem:abs1} is a simple $k$-partite version of \cite[Fact 2.3]{RRS09}; Lemma~\ref{lem:abs2} is derived from a more involved approach by considering the lattice generated by the edges of $H$. In Sections~3--5, we give the proofs of Theorems~\ref{thm:nonext}--\ref{thm:ext2}, respectively. Note that Lemma~\ref{lem:abs1}, Lemma~\ref{lem:alm_mat}, and a portion of Section 4 suffice for the proof of Corollary~\ref{cor} -- this was exactly the approach used in \cite{Han14_mat, LWY}. The rest of our proof was carried through with new ideas.

\medskip
\noindent\textbf{Notation}:
Given integers $k'\ge k\ge 1$, we write $[k]:=\{1,\dots, k\}$ and $[k, k']:=\{k, k+1,\dots, k'\}$.
Throughout this paper, we denote by $H$ a $k$-partite $k$-graph with the vertex partition $V(H)=V_1\cup \cdots \cup V_k$. 
A vertex set $S$ is called \emph{balanced} if it consists of an equal number of vertices from each part of $V(H)$.
Given a $k$-graph $H$ and a set $W\subseteq V(H)$, let $H[W]$ denote the subgraph of $H$ induced on $W$ and
$H\setminus W := H[V(H)\setminus W]$. 

\section{Absorbing Techniques in $k$-Partite $k$-Graphs}

The main tool in the proof of Theorem~\ref{thm:nonext} is the absorbing method.
This technique was initiated  by R\"odl, Ruci\'nski and Szemer\'edi \cite{RRS06} and has proven to be a powerful tool for finding spanning structures in graphs and hypergraphs.
In this section, we prove the absorbing lemmas that will be used in the proof of Theorem \ref{thm:nonext}.
In fact, we present two different notions of absorbing sets and use them in two different cases.

Let $H$ be a $k$-partite $k$-graph.
Given a balanced $2k$-set $S$, an edge $e\in E(H)$ disjoint from $S$ is called \emph{$S$-absorbing} if there are two disjoint edges $e_1, e_2\subseteq S\cup \{e\}$ such that $|e_1\cap S|=k-1$, $|e_1\cap e|=1$, $|e_2\cap S|=2$, and $|e_2\cap e|=k-2$.
Note that $S$-absorbing works in the following way: assume that $M$ is a matching such that $S\cap V(M)=\emptyset$ and $M$ contains an $S$-absorbing edge $e$, then we can replace $e$ by $e_1$ and $e_2$ and get a matching larger than $M$.
Given a crossing $k$-set $S$, a set $T\subset V(H)\setminus S$ is called \emph{$S$-perfect-absorbing} if $T$ is balanced and both $H[T]$ and $H[S\cup T]$ contain perfect matchings.
These two definitions work very differently -- they are needed for the following two different absorbing lemmas.

Our first absorbing lemma is an analog of~\cite[Fact 2.3]{RRS09} for $k$-partite $k$-graphs. 

\begin{lemma}[Absorbing lemma I]\label{lem:abs1}
Let $1/n \ll \a \ll \e \ll 1/k$. 
Suppose H is a k-partite k-graph with parts of size $n$ such that $\delta_{[k]\setminus\{i\}}(H) \ge \e n$ for $i\in [3]$, then there exists a matching $M'$ in $H$ of size at most $\sqrt\a n$ such that for every balanced $2k$-set $S$ of $H$, the number of $S$-absorbing edges in $M'$ is at least $\a n$.
\end{lemma}

Our second absorbing lemma deals with the case when only two partite minimum codegrees are large and their sum is not significantly smaller than $n$.

\begin{lemma}[Absorbing Lemma II]\label{lem:abs2}
Let $1/n \ll \a \ll \e \ll 1/t \ll 1/k$.
Suppose $H$ is a $k$-partite $k$-graph with parts of size $n$ and $a_i:=\delta_{[k]\setminus\{i\}}(H)$ for each $i\in [k]$.
If $\sum_{i=1}^k a_i \ge (1-\e)n$, $a_1\ge a_2 \ge \e n$ and $a_j<\e n$ for $j\ge 3$, then one of the following holds.
\begin{itemize}
\item[(i)]  $H$ is $2k^2\e$-D-extremal.
\item[(ii)]
There exists a family $\F'$ of disjoint $t k$-sets such that $|\F'|\le \sqrt\a n$, each $F\in \F'$ spans a matching of size $t$ and for every crossing $k$-set $S$ of $H$, the number of $S$-perfect-absorbing sets in $\F'$ is at least $\a n$.
\end{itemize}
\end{lemma}

We first prove the following proposition, which is a standard application of Chernoff's bound.
We will apply it in both proofs of Lemmas~\ref{lem:abs1} and~\ref{lem:abs2} for randomly sampling the absorbing sets.

\begin{proposition}\label{prop:prob}
Let $1/n \ll \lambda, 1/k, 1/i_0$.
Let $V$ be a vertex set with $k$ parts each of size $n$, and let $\F_1, \dots, \F_t$ be families of balanced $i_0 k$-sets on $V$ such that $|\F_i|\ge \lambda n^{i_0k}$ for $i\in [t]$ and $t\le n^{2k}$. 
Then there exists a family $\F'\subseteq \bigcup_{i\in [t]} \F_i$ of disjoint balanced $i_0 k$-sets on $V$ such that $|\F'|\le \lambda n/(4i_0 k)$ and $|\F_i\cap \F'|\ge \lambda^2 n/(32i_0 k)$ for each $i\in [t]$.
\end{proposition}

\begin{proof}
We build $\F'$ by standard probabilistic arguments. 
Choose a collection $\F$ of balanced $i_0 k$-sets in $H$ by selecting each balanced $i_0 k$-set on $V$ independently and randomly with probability $p= {\e}/(2n^{i_0k-1})$, where $\e=\lambda/(4i_0 k)$.  
Since $t\le n^{2k}$, Chernoff's bound implies that, with probability $1-o(1)$, the family $\F$ satisfies the following properties:
\[
|\F|\le 2p \binom{n}{i_0}^k \le 2n^{i_0k}p={\e} n \quad \text{and} \quad |\F_i\cap \F|\ge \frac p2\cdot \lambda n^{i_0k} =\frac14\lambda{\e} n \text{ for any }i\in [t].
\]
Furthermore, the expected number of intersecting pairs of members in $\F$ is at most
\[
p^2 n^{i_0 k} \cdot i_0 k \cdot n^{i_0 k-1} =\e^2 i_0 k n/4.
\]
By Markov's inequality, $\F$ contains at most ${\e}^2 i_0 k n/2$ intersecting pairs of $i_0 k$-sets with probability at least $1/2$.

Let $\F'\subset \F$ be the subfamily obtained by deleting one $i_0 k$-set from each intersecting pair and removing all $i_0 k$-sets that do not belong to any $\F_i$, $i\in [t]$.
Therefore, $|\F'|\le |\F|\le \e n$ and for each $i\in [t]$, we have
\[
|\F_i\cap \F'|\ge |\F_i\cap \F| - \frac12{\e}^2 i_0 kn \ge \frac14\lambda {\e} n -\frac12{\e}^2 i_0 k n = \frac{\lambda^2}{32 i_0 k}n
\] 
and we are done.
\end{proof}

Now we prove our first absorbing lemma. 


\begin{proof}[Proof of Lemma~\ref{lem:abs1}]\label{lem:abs}
We claim that for every balanced $2k$-set $S$, there are at least ${\e}^3n^k/2$ $S$-absorbing edges. Since there are at most $n^{2k}$ balanced $2k$-sets, the existence of the desired matching follows from Proposition~\ref{prop:prob}.

Indeed, assume that $\{w,v\}:=S\cap V_3$ and $u \in S\cap V_{2}$.
We obtain $S$-absorbing edges $e=\{v_{1}, v_{2}, \dots, v_{k}\}$ as follows.
First, for each $j\in [4, k]$, we choose arbitrary $v_i\in V_j\setminus S$ -- there are $n-2$ choices for each $v_{j}$.
Having selected $\{v_{4}, v_{5}, \dots, v_{k}\}$, 
we select a neighbor of $\{u, v, v_{4}, \dots, v_{k}\}$ as $v_{1}$. 
Next, we choose a neighbor of $S'$ as $v_2$, where $S'$ is an arbitrary crossing $(k-1)$-subset of $S\setminus V_2$ that
contains $w$. Finally, we choose a neighbor of $\{v_{1}, v_{2}, v_{4}, \dots, v_{k} \}$ as $v_{3}$. There are at least $\e n-2$ choices for $v_{j}$ for $j = 1, 2, 3$. Hence, there are at least
\[
(n-2)^{k-3}(\e n-2)^3\ge \frac12 {\e}^3 n^k > \sqrt{32k\a} n^k
\]
$S$-absorbing edges, since $n$ is sufficiently large and $\a \ll \e$.
Then we get the absorbing matching $M'$ by applying Proposition \ref{prop:prob} with $\lambda = \sqrt{32 k\a}$ and $i_0=1$.
\end{proof}

The proof of Lemma~\ref{lem:abs2} is more involved than that of Lemma~\ref{lem:abs1} -- we need to apply a \emph{lattice-based absorbing method}, a variant of the absorbing method developed recently by the first author~\cite{Han14_poly}. 
Roughly speaking, the method provides a vertex partition $\cP$ of $H$ (Lemma~\ref{lem:P}) which refines the original $k$-partition so that we can work on the vectors of $\{0, 1\}^{|\cP|}$ that represent the edges of $H$.
Using the information obtained from these vectors, we will show that if Lemma~\ref{lem:abs2} (ii) does not hold, then $H$ is close to a divisibility barrier based on $\cP$.
The rest of this section is devoted to the proof of Lemma~\ref{lem:abs2}, for which we need the following notation and auxiliary results.

The following concepts are introduced by Lo and Markstr\"om~\cite{LM1}.
Let $H$ be a $k$-partite $k$-graph with $n$ vertices in each part.
Given $\b>0$, $i \in \mathbb{N}, j\in [k]$ and two vertices $u,v\in V_j$, we say that $u, v$ are \emph{$(\b ,i)$-reachable in $H$} if and only if there are at least $\b n^{ik-1}$ $(ik-1)$-sets $W$ such that both $H[\{u\} \cup W]$ and $H[\{v\} \cup W]$ contain perfect matchings.
In this case $W$ is called a \emph{reachable set} for $u, v$.
A set $X\subseteq V_j$ is \emph{$(\b,i)$-closed in $H$} if all $u,v\in X$ are $(\b ,i)$-reachable in $H$. Denote by $\tilde{N}_{\b, i}(v)$ the set of vertices that are $(\b ,i)$-reachable to $v$ in $H$.
Clearly, since $H$ is $k$-partite, for any $j\in [k]$ and $v\in V_j$, $\tilde{N}_{\b, i}(v)\subseteq V_j$.

We need the following simple fact on $k$-partite $k$-graphs.

\begin{fact}\label{degree}
Let H be a k-partite k-graph with parts of size $n$. Let $a_1:=\delta_{[k]\setminus\{1\}}(H)$.
\begin{itemize}
\item[(i)] For any $i\in [2,k]$ and $v\in V_i$, we have $\deg(v)\ge a_1 n^{k-2}$.
\item[(ii)] If $a_1\ge (1/3+\r)n$, then for any $i\in [2, k]$, any set of three vertices $u, v, w\in V_i$ contains a pair of vertices which are $(\r, 1)$-reachable.
\end{itemize}
\end{fact}

\begin{proof} 
To see (i), note that we can obtain an edge containing $v$ by first choosing a $(k-2)$-set $S\in \Pi_{j\ne 1, i} V_j$, and then choosing a neighbor of $\{v\} \cup S$. 
To see (ii), by (i) and $a_1\ge (1/3+\r)n$, we have $|N(u)|, |N(v)|, |N(w)|\ge (1/3+\r)n^{k-1}$.
Also note that $|N(u)\cup N(v)\cup N(w)|\le n^{k-1}$, then by the inclusion-exclusion principle, we have
\[
n^{k-1} \ge |N(u)| + |N(v)| + |N(w)|- |N(u)\cap N(v)| - |N(u)\cap N(w)| - |N(v)\cap N(w)|.
\]
So we get $|N(u)\cap N(v)| + |N(u)\cap N(w)| + |N(v)\cap N(w)| \ge 3\r n^{k-1}$.
Without loss of generality, assume that $|N(u)\cap N(v)|\ge \r n^{k-1}$.
This implies that $u$ and $v$ are $(\r, 1)$-reachable. 
\end{proof}

The following proposition reflects the property of $|\tilde{N}_{\e, 1}(v)|$.

\begin{proposition}\label{prop:Nv}
Suppose $1/n \ll \e \ll 1/k$ and let $H$ be a $k$-partite $k$-graph with $n$ vertices in each part such that $\delta_{[k]\setminus \{1\}}(H), \delta_{[k]\setminus \{2\}}(H) \ge (1/2-\e) n$.
Then for any $j=1,2$ and $v\in V_j$, $|\tilde{N}_{\e/3, 1}(v)| \ge (1/2 - 2\e) n$.
Moreover, for each $j\in [3,k]$, either $|\tilde{N}_{\e/3, 1}(v)| \ge \e n$ holds for all vertices $v\in V_j$, or there exists a set $V_j'\subseteq V_j$ of size at most $\e n+1$ such that $V_j\setminus V_j'$ is $(\e/3, 1)$-closed in $H$.
\end{proposition}

\begin{proof}
Fix a vertex $v\in V_j$ for some $j=1,2$, note that for any other vertex $u\in V_j$, $u\in \tilde{N}_{\e/3, 1}(v)$ if and only if $|N(u)\cap N(v)|\ge \e n^{k-1}/3$.
By double counting, we have
\[
\sum_{S\in N(v)} (\deg(S) -1) < |\tilde{N}_{\e/3, 1}(v)|\cdot |N(v)|+n\cdot \e {n}^{k-1}/3.
\]
Note that $ \sum_{S\in N(v)} (\deg(S)-1) \ge ((1/2-\e) n-1) |N(v)|$. Moreover, by Fact~\ref{degree} (i), $|N(v)|\ge (1/2-\e)n^{k-1}> 2n^{k-1}/5$. Putting these together, we conclude that  
\[
|\tilde{N}_{\e/3, 1}(v)|> \left(\frac12-\e \right) n -1- \frac{\e n^k/3}{|N(v)|}\ge \left(\frac12-\e \right) n - \e n = \left(\frac12 - 2\e \right) n.
\]

Now assume $j\in [3,k]$ and assume that $|\tilde{N}_{\e/3, 1}(v)| < \e n$ for some $v\in V_j$.
Let $V_j':=\{v\}\cup \tilde{N}_{\e/3, 1}(v)$. Thus $|V_j'|\le \e n+1$.
For any $u, u'\in V_j\setminus V_j'$, since $u\notin \tilde{N}_{\e/3, 1}(v)$ and $u'\notin \tilde{N}_{\e/3, 1}(v)$, by Fact~\ref{degree} (ii), we conclude that $u$ and $u'$ are $(\e/3, 1)$-reachable.
This implies that $V_j\setminus V_j'$ is $(\e/3, 1)$-closed.
\end{proof}

We use the following lemma from~\cite{HT} to find a partition of each part of $H$.

\begin{lemma}\cite[Lemma 6.3]{HT}\label{lem:P}
Let $1/m \ll \beta \ll\r \ll 1/c, \delta', 1/k$.
Suppose that $H$ is an $m$-vertex $k$-graph, and a subset $S\subseteq V(H)$ satisfies that $|\tilde{N}_{\r, 1}(v)| \ge \delta' m$ for any $v\in S$ and every set of $c+1$ vertices in $S$ contains at least two vertices that are $(\r, 1)$-reachable.
Then we can find a partition $\cP_0$ of $S$ into $W_1,\dots, W_d$ with $d\le \min\{\lfloor 1/\delta' \rfloor, c\}$ such that for any $i\in [d]$, $|W_i|\ge (\delta' - \r) m$ and $W_i$ is $(\beta, 2^{c-1})$-closed in $H$.
\end{lemma}

The following useful proposition was proved in~\cite{LM1}.

\begin{proposition}\cite[Proposition 2.1]{LM1}\label{prop21}
Let $i\ge 1$, $k\ge 2$ and $1/n \ll \beta \ll\e, \beta, 1/i, 1/k$. 
Suppose $H$ is a $k$-graph of order $n\ge n_0$ and there exists a vertex $x\in V(H)$ with $|\tilde{N}_{\beta, i}(x)|\ge \e n$. Then for all $0<\beta'\le \beta_0$, $\tilde{N}_{\beta,i}(x)\subseteq \tilde{N}_{\beta',i+1}(x)$.
\end{proposition}

Let $H$ be a $k$-partite $k$-graph with parts of size $n$.
Suppose $\cP=\{W_0, W_1, \dots, W_d\}$ is a partition  of $V(H)$ for some integer $d\ge k$ that refines the original $k$-partition of $H$.
In later applications, $W_0$ will be so small that we only need to consider the edges not intersecting $W_0$.
The following concepts were introduced by Keevash and Mycroft~\cite{KM1}. The \emph{index vector} of a subset $S\subseteq V(H)$ with respect to $\cP$ is the vector 
\[
 \bfi_{\cP}(S) := \left( |S\cap W_1|, \dots, |S\cap W_d| \right) \in \mathbb{Z}^d.
\]
Given an index vector $\bfv$, we denote by $\bfv|_{W_i}$ its value at the coordinate that corresponds to $W_i$.
For $\mu>0$, define $I_{\cP}^\mu(H)$ to be the set of all $\bfv \in \mathbb{Z}^d$ such that $H$ contains at least $\mu n^k$ edges with index vector $\bfv$; let $L_{\cP}^{\mu}(H)$ denote the lattice in $\mathbb{Z}^d$ generated by $I_{\cP}^\mu(H)$.
For $i\in [d]$, let $\bfu_{W_i}\in \mathbb{Z}^d$ be the \emph{unit vector} such that $\bfu_{W_i}|_{W_i}=1$ and $\bfu_{W_i}|_{W_j}=0$ for $j\neq i$.

\begin{lemma}\cite[Lemma 3.4]{Han15_mat}\label{lem:trans}
Suppose $\beta' \ll\mu, \beta\ll \e\ll 1/i_0, 1/k$ and $1/i' \ll 1/i_0, 1/k$, then the following holds for sufficiently large $m$.
Suppose $H$ is an $m$-vertex $k$-graph and $\cP=\{W_0, W_1, \dots, W_d\}$ is a partition with $d\le 2k$ such that $|W_0|\le \sqrt\e m$, $|W_i|\ge \e^2 m$ and $W_i$ is $(\beta, i_0)$-closed in $H$ for $i\ge 1$.
If $\bfu_{W_i} - \bfu_{W_j}\in L_{\cP}^{\mu}(H)$, then $W_i\cup W_j$ is $(\beta', i')$-closed in $H$.
\end{lemma}

The following lemma shows that if $V_1$ is closed and $\delta_{[k]\setminus \{1\}}(H)\ge \e n$ then Lemma~\ref{lem:abs2} (ii) holds.
\begin{lemma}\label{claim:abs}
Let $1/n\ll \a \ll\beta, \e, 1/i_0, 1/k$.
Suppose $H$ is a $k$-partite $k$-graph with parts each of size $n$ and $\delta_{[k]\setminus \{1\}}(H)\ge \e n$.
If $V_1$ is $(\b, i_0)$-closed, then there exists a family of disjoint $i_0 k$-sets $\F'$ in $H$ such that $|\F'|\le \sqrt\a n$ and each $F\in \F'$ spans a matching of size $i_0$ and for every crossing $k$-set $S$ of $H$, the number of $S$-perfect-absorbing sets in $\F'$ is at least $\a n$.
\end{lemma}

\begin{proof}

Fix a crossing $k$-set $S=\{v_{1}, v_{2}, \dots, v_{k}\}$ such that $v_{j}\in V_{j}$, we claim there are at least $\sqrt{32i_0 k\a} n^{i_0 k}$ $S$-perfect-absorbing $i_0k$-sets.
First of all, we find $v_1'\in V_1\setminus\{v_1\}$ such that $\{v'_{1}, v_{2}, \dots, v_{k}\}$ spans an edge.
Since $\deg(S\setminus\{v_1\})\ge \e n$, there are at least $\e n-1$ choices for $v_1'$.
Since $V_1$ is $(\b, i_0)$-closed, there are at least $\b n^{i_0 k-1}$ reachable $(i_0 k-1)$-sets $W$ for $v_1$ and $v_1'$.
Among them, at least $\b n^{i_0k-1} - (k-1)n^{k-2}\ge \b n^{i_0k-1}/2$ reachable $(i_0k-1)$-sets $W$ are disjoint from $S$.
To see that $\{v_1'\}\cup W$ is an $S$-perfect-absorbing set, note that $H[\{v_1'\}\cup W]$ has a perfect matching by the definition of $W$, and $H[\{v_1'\}\cup W\cup S]$ has a perfect matching because $\{v_1'\}\cup (S\setminus \{v_1\})$ spans an edge and $H[\{v_1\}\cup W]$ has a perfect matching by the definition of $W$.
In total, we have at least $\e \b n^{i_0k}/4 > \sqrt{32i_0 k\a} n^{i_0 k}$ $S$-perfect-absorbing sets.
So we get the family of absorbing sets $\F'$ by applying Proposition \ref{prop:prob} with $\lambda=\sqrt{32 i_0 k\a}$.
Note that each $F\in \F'$ is an $S$-perfect-absorbing set for some crossing $k$-set $S$ and thus $F$ spans a matching of size $i_0$.
\end{proof}

\begin{proof}[Proof of Lemma~\ref{lem:abs2}]
We apply Lemma~\ref{lem:trans} inductively $k$ times, at the $j$th time with $i = t_{j-1}$ and $i' = t_j$, where $t_0=2$.
Let $t=t_{k}$.
Pick further constants such that
\[
1/n\ll \a \ll \beta_k \ll \beta_{k-1}\ll \cdots \ll \beta_1 \ll \mu, \beta \ll \e \ll 1/t_k \ll 1/k.
\]

Let $H$ be a $k$-partite $k$-graph as given by Lemma~\ref{lem:abs2}.
Suppose that (ii) does not hold.
In particular, by Lemma~\ref{claim:abs}, we may assume that neither $V_1$ nor $V_2$ is $(\beta_k, t_k)$-closed in $H$.
By Fact \ref{degree} (i), we have $\deg(v) \ge a_1 n^{k-2}$ for any $v\notin V_1$, and $\deg(v)\ge a_2 n^{k-2}$ for any $v\in V_1$.

First note that if $a_1\ge (1/2+\e)n$, then for any $u, v\in V_2$, we have $|N(u)\cap N(v)|\ge 2\e n^{k-1}$, and thus $V_2$ is $(2\e, 1)$-closed. By Proposition~\ref{prop21}, $V_2$ is $(\beta_k, t_k)$-closed, a contradiction.

So we may assume that $a_1 < (1/2+\e)n$.
Thus, we have
\[
a_2\ge \sum_{i=1}^k a_i - a_1 - (k-2) \e n \ge (1-\e)n - (1/2+\e)n - (k-2)\e n = (1/2-k\e)n,
\]
i.e., $(1/2-k\e) n \le a_2\le a_1 < (1/2+\e)n$.
Let $\r:=(k-1)\e /k$.
We apply Proposition~\ref{prop:Nv} with $k\e$ in place of $\e$ and obtain that, using $\r \le \e \le k\e/3$, 
\begin{enumerate}
\item for any $i=1,2$ and $v\in V_i$, $|\tilde{N}_{\r, 1}(v)| \ge (1/2 - 2k\e) n$,
\item for any $i\in [3,k]$, either $|\tilde{N}_{\r, 1}(v)| \ge k\e n$ for all vertices $v\in V_i$, or there exists a set $V_i'\subseteq V_i$ of size at most $k\e n+1$ such that $V_i''=V_i\setminus V_i'$ is $(\r, 1)$-closed in $H$.
\end{enumerate}

Since $a_1, a_2\ge (1/2-k\e)n\ge (1/3 + \r) n$, Fact~\ref{degree} (ii) implies that for \emph{any} $i\in [k]$, 
every set of three vertices of $V_i$ contains two vertices that are $(\r,1)$-reachable in $H$.
Together with (1), it allows us to apply Lemma~\ref{lem:P} to $V_1$, $V_2$ separately with $c=2$ and $\delta'=1/(2k) - 2\e$ and partition each of $V_1$ and $V_2$ into at most two parts such that each part is $(\beta, 2)$-closed. If $V_1$ or $V_2$ is $(\beta, 2)$-closed, then by Proposition~\ref{prop21}, it is $(\beta_k, t_k)$-closed, a contradiction. Thus we assume that each of $V_1$ and $V_2$ is partitioned into two parts $V_1= A_1\cup B_1$ and $V_2=A_2\cup B_2$ such that $A_i, B_i$ are $(\beta, 2)$-closed, and 
\[
 |A_i|, |B_i| \ge (\tfrac1{2k} - 2\e - \r) kn > (\tfrac12 - 3k\e)n. 
\]
Without loss of generality, assume that $|A_1|\le |B_1|$ and $|A_2|\le |B_2|$.

Let $I$ be the set of $i\in [3,k]$ such that $|\tilde{N}_{\r, 1}(v)| \ge k\e n$ for all vertices $v\in V_i$, and let $I'\subseteq I$ consist of those $i\in I$ such that $V_i$ is not $(\beta, 2)$-closed. We now apply Lemma~\ref{lem:P} to $V_i$ for $i\in I'$ with $c=2$ and $\delta'=\e$ and partition $V_i$ into at most two parts such that each part is of size at least $(\e - \r) kn= \e n$ and is $(\beta, 2)$-closed. Since $V_i$, $i\in I'$, is not $(\beta, 2)$-closed, it must be the case that $V_i$ is partitioned into two parts, denoted by $A_i$ and $B_i$.  
Let $W_0=\bigcup_{i\in [3,k]\setminus I} V_i'$ and note that $|W_0|\le (k-2) (k\e n+1)$. 
Let $\cP_0$
be the partition of $V(H)$ consisting of $W_0, A_1, B_1, A_2, B_2$ and $V''_i$ if $i\in [3,k]\setminus I$, $V_i$ if $i\in I\setminus I'$, or $A_i$, $B_i$ if $i\in I'$.
By Proposition~\ref{prop21}, each part of $\cP_0$ except $W_0$ is $(\beta, 2)$-closed. 

For $i\in [k]$ for which $A_i$ and $B_i$ were defined, if $\bfu_{A_i} - \bfu_{B_i} \in L_{\cP_0}^{\mu}(H)$, then we \emph{merge} $A_i$ and $B_i$ by replacing $A_i$ and $B_i$ with $V_i$.
By Lemma~\ref{lem:trans}, $V_i = A_i\cup B_i$ is $(\beta_1, t_1)$-closed.
We inductively merge $A_i$, $B_i$ as long as $\bfu_{A_i} - \bfu_{B_i} \in L_{\cP'}^{\mu}(H)$, where $\cP'$ represents the current partition after merging some parts.
Since neither $V_1$ nor $V_2$ is $(\beta_k, t_k)$-closed in $H$, $A_1$ and $B_1$ (also $A_2$ and $B_2$) cannot be merged. 
After at most $k-2$ merges, we obtain a partition $\cP$ such that each part except $W_0$ is $(\beta_k, t_k)$-closed (by Proposition~\ref{prop21}).
Write $\cP:=\{W_0, W_1,\dots, W_d \}$.
Let $\tilde{I}\subseteq [k]$ be the set of $i$ such that $A_i, B_i \in \cP$ (note that $\bfu_{A_i} - \bfu_{B_i} \notin L_{\cP}^{\mu}(H)$ for $i\in \tilde{I}$).
Write $T:=I_{\cP}^{\mu}(H) \subseteq  \{0, 1\}^d$. Given $i, j \in \tilde{I}$ and $\bfw\in T$, let $\bfw^i: = \bfw + \bfu_{A_i} - \bfu_{B_i} \pmod 2$ and $\bfw^{i, j} := \bfw + \bfu_{A_i} - \bfu_{B_i} + \bfu_{A_j} - \bfu_{B_j} \pmod 2$.
We have the following observations.
\begin{enumerate}
\item[($\dagger$)] If $\bfw\in T$, then $\bfw^i \notin T$ for $i\in \tilde{I}$.
\item[($\ddagger$)] If $\bfw\in T$, then $\bfw^{i, 1}\in T$ for $i\in \tilde{I}$.
\end{enumerate}
Indeed, for ($\dagger$), if $\bfw\in T$, then $\bfw^i \notin T$ because $\bfu_{A_i} - \bfu_{B_i}\notin L_{\cP}^{\mu}(H)$ for $i\in \tilde{I}$.
For ($\ddagger$), note that $\bfw^i$ has $1$ at $k$ coordinates, which correspond to $W_j\subseteq V_j$, $j\in [k]$ of size at least $\e n$ (where $W_j$ is $V_j$ or $V''_j$ or $A_j$ or $B_j$).
The number of the edges in $H$ that contain a crossing $(k-1)$-set in $\prod_{j\in [2,k]} C_j$ is at least $(\e n)^{k-1} a_1$.
Since $\bfw^i \notin T$, there are at most $\mu n^k$ edges $e$ in $H$ with $\bfi_{\cP}(e)=\bfw^i$. 
Consequently, the number of edges $e$ with $\bfi_{\cP}(e) = \bfw^{i, 1}$ is at least $(\e n)^{k-1} a_1 - \mu n^k \ge \mu n^k$, because $\mu \ll \e$ and $a_1\ge n/3$. Hence $\bfw^{i, 1} \in T$ and this proves ($\ddagger$).

A vector $\bfv\in \{0,1\}^{d}$ is \emph{even} (respectively, \emph{odd}) if there is an even (respectively, odd) number of $i\in \tilde{I}$ such that $\bfv |_{A_i}=1$. 
We claim that all the vectors in $T$ have the same parity. Indeed, assume that there is an even vector $\bfv \in T$. By ($\ddagger$), we know that all even vectors are in $T$. Together with ($\dagger$), this implies that $T$ contains no odd vector.

Assume that $T$ only contains even vectors (the case when $T$ only contains odd vectors is analogous).
Let $A :=\bigcup_{i\in \tilde{I}}A_i$ and $B:= V\setminus A$. Recall that an edge $e$ of $H$ is {even} if $|e\cap A |$ is even. Since $T$ only contains even vectors, $E(H\setminus W_0)$ contains at most $2^{k}\mu n^k$ odd edges. 
Recall that $(1/2 - 3k\e)n\le |A_1|, |A_2| \le n/2$.
In addition, we have shown that $\deg(v)\ge (1/2-k\e)n^{k-1}$ for all $v\in V(H)$ and thus $|E(H)|\ge (1/2-k\e)n^k$.
Since $\mu \ll \e$ and $|W_0|\le (k-2)(k\e n+1)$, there are at least 
\[
|E(H)| - 2^{k}\mu n^k - |W_0|n^{k-1}\ge (1/2-k\e)n^k - 2^{k}\mu n^k - (k-2)(k\e n+1) n^{k-1} \ge (1/2 - k^2\e)n^k
\]
even edges in $E(H)$.
Let $|A_1|=n/2-y$ for some $0\le y\le 3k\e n$ and assume that the number of odd crossing $(k-1)$-sets in $V\setminus V_1$ is $x$ for some $0\le x\le n^{k-1}$, then we get
\begin{align*}
|E_{even}(A, B)| &= (n/2-y) x + (n/2+y) (n^{k-1} - x) \\
&= n^k/2 + y(n^{k-1} - 2x) \le n^k/2 + 3k\e n^k.
\end{align*}
Thus we have $|E_{even}(A, B)\setminus E(H)|\le n^k/2 + 3k\e n^k - (1/2-k^2\e)n^k \le 2k^2\e n^k$.
Together with $(1/2 - 3k\e)n\le |A_1|, |A_2| \le n/2$ and $(1/2-k\e) n \le a_2\le a_1 < (1/2+\e)n$, and $a_i < \e n$, $3\le i\le k$, we conclude that $H$ is $2k^2\e$-D-extremal.
This completes our proof.
\end{proof}

\section{Nonextremal $k$-partite $k$-graphs: proof of Theorem \ref{thm:nonext}}

In this section we first show that every $k$-partite $k$-graph $H$ contains an almost perfect matching if $\sum a_i$ is near $n$ and $H$ is not close to $H_0$.  The following lemma is an analog of \cite[Lemma 1.7]{Han14_mat} in $k$-partite $k$-graphs.
To make it applicable to other problems, we prove it under a weaker assumption which allows a small fraction of crossing $(k-1)$-sets to have small degree. 

\begin{lemma}[Almost perfect matching]\label{lem:alm_mat}
Let $1/n\ll \eta \ll \a, \r, 1/k$.
For $i\in [k]$, let $a_i=a_i(n)$ such that $\sum_{i=1}^k a_i \ge (1-\r) n$.
Let $H$ be a $k$-partite $k$-graph with parts of size $n$ which is not $2\r$-S-extremal.
Suppose for each $i\in [k]$, there are at most $\eta n^{k-1}$ crossing $(k-1)$-sets $S$ such that $S\cap V_i=\emptyset$ and $\deg(S)< a_i$.
Then $H$ contains a matching that covers all but at most $\a n$ vertices in each vertex class.
\end{lemma}

\begin{proof}
Let $M$ be a maximum matching in $H$ and assume $m=|M|$. Let $V_i'=V_i\cap V(M)$ and $U_i=V_i\setminus V(M)$. Suppose to the contrary, that $|U_1| = \cdots = |U_k| > \a n$.

Let $t = \lceil k(k-1)/\r \rceil$.
We find a family $\A$ of disjoint crossing $(k-1)$-subsets $A_1, \dots, A_{k t}$ of $V\setminus V(M)$ such that $A_j \cap V_{i}=\emptyset$ and $\deg(A_j)\ge a_i$ whenever $j \equiv i \bmod k$.
This can be done greedily because when selecting $A_j$, the crossing $(k-1)$-sets that cannot be picked are either those that intersect the ones that have been picked, or those with low degree, whose number is at most
\[
k (k-1)t n^{k-2} + \eta n^{k-1} < (\a n)^{k-1} < \prod_{\ell \in [k]\setminus \{i\}} |U_{\ell}|,
\]
because $1/n\ll\eta \ll \a$. Note that the neighbors of $A_j$ are in $V'_i$ with $j \equiv i \bmod k$ by the maximality of $M$.

For $i\in[k]$, let $D_i$ be the set of the vertices of $V_i'$ that have at least $k$ neighbors in $\A$ and let $D=\bigcup D_i$.
We claim that $|e\cap D|\le 1$ for each $e\in M$. Indeed, otherwise assume that $x, y\in e\cap D$ and pick $A_i, A_{j}$ for some $i, j\in [k t]$ such that $\{x\}\cup A_i, \{y\}\cup A_j\in E(H)$. We obtain a matching of size $m+1$ by deleting $e$ and adding $\{x\}\cup A_i$ as well as $\{y\}\cup A_j$ in $M$, contradicting the maximality of $M$.

Next we show that $|D_i|\ge a_i - \r n/k$ for each $i\in [k]$. Since $N(A_j)\cap V'_{i} = \emptyset$ for $j\not\equiv i$ mod $k$, 
we get
\[
t\cdot a_i \le \sum_{j\equiv i \text{ mod } k} \deg(A_j) \le |D_i| t +n (k-1).
\]
Since $t\ge k(k-1)/\r$, it follows that
\[
|D_i|\ge a_i - \frac{n (k-1)}{t} \ge a_i - \frac{\r n}{k}.
\]
This implies that $|D| = \sum_{i=1}^k |D_i| \ge \sum_{i=1}^k a_i - \r n$.
Since every edge of $M$ contains at most one vertex of $D$, we have $|D|\le |M| < n$ and consequently, $\sum_{i=1}^k a_i \le n + \r n = (1+\r)n$.

Define $M':=\{e\in M: e\cap D\neq \emptyset\}$. Then for each $i\in [k]$, we have
\[
|(V(M')\setminus D)\cap V_i| = \sum_{j\ne i}|D_j| \ge \sum_{j\in [k]} (a_j - \tfrac{\r n}{k}) - a_i \ge n-a_i - 2\r n.
\]
Since $H$ is not $2\r$-S-extremal, $H[V(M')\setminus D]$ contains at least one edge, denoted by $e_0$. Note that $e_0\not\in M$ because each edge of $M'$ contains exactly one vertex of $D$ and $e_0\subset V(M')\setminus D$.
Assume that $e_0$ intersects $e_{1}, \dots, e_{p}$ in $M$ for some $2\le p\le k$.  Suppose $\{v_{j}\} := e_{j}\cap D$. Note that $v_j\not\in e_0$ for all $j\in [p]$. Since each $v_j$ has at least $k$ neighbors in $\A$, we can greedily pick $A_{\ell_1}, \dots, A_{\ell_p}\in \A$ such that  $\{v_{j}\}\cup A_{\ell_j} \in E(H)$ for all $j\in [p]$. Let $M''$ be the matching obtained from $M$ after replacing $e_{1}, \dots, e_{p}$ by $e_0$ and $\{v_{j}\}\cup A_{\ell_j}$ for $j\in [p]$. Thus, $M''$ has $m+1$ edges, contradicting the choice of $M$.
\end{proof}

Now we are ready to prove Theorem \ref{thm:nonext}.
%

\begin{proof}[Proof of Theorem \ref{thm:nonext}]
Let $1/n \ll \eta \ll \a \ll \r \ll \e \ll 1/t \ll 1/k$.
Suppose both (ii) and (iii) fail and we will show that (i) holds. 

First assume that $a_1\ge a_2\ge a_3\ge \e n$. We first apply Lemma~\ref{lem:abs1} to $H$ and find a matching $M'$ of size at most $\sqrt\a n$ such that for every balanced $2k$-set $S\subset V(H)$, the number of $S$-absorbing edges in $M'$ is at least $\a n$.
Let $H':=H\setminus V(M')$, $n':=|V(H')\cap V_i|\ge (1-\sqrt\a)n$ and $a'_i := \delta_{[k]\setminus \{i\}}(H')$.
Note that $\sum_{i=1}^k a'_i \ge \sum_{i=1}^k a_i - k\sqrt{\a} n \ge (1 - 2\r/5) n'$.
Assume for a moment that $H'$ is $(4\r/5)$-S-extremal, i.e., $\sum_{i=1}^k a_i'\le n' + (4\r/5) n'$ and $V(H')$ contains an independent set $C$ such that $|C\cap (V_i\cap V(H'))| \ge n'-a_i'-4\r n'/5$ for each $i\in [k]$.
Then as $\a\ll \r$, it follows that $\sum_{i=1}^k a_i \le \sum_{i=1}^k a_i' + k\sqrt\a n \le n' + (4\r/5) n' + k\sqrt\a n\le n+ \r n$ and
\[ 
n'-a_i'-4\r n'/5 \ge (1- \sqrt{\a})n - a_i - 4\r n/5 \ge n - a_i - \r n,
\]
This means that $H$ is $\r$-S-extremal, a contradiction. 
Thus, $H'$ is not $(4\r/5)$-S-extremal. By applying Lemma \ref{lem:alm_mat} to $H'$ with parameters $2\r/5$, $\a$ and $\eta$, we obtain a matching $M''$ in $H'$ that covers all but at most $\a n$ vertices in each vertex class.

Since there are at least $\a n$ $S$-absorbing edges in $M'$ for every balanced $2k$-set $S\subset V(H)$, we can repeatedly absorb the leftover vertices until there is one vertex left in each class.
Denote by $\tilde{M}$ the matching obtained after absorbing the leftover vertices into $M'$. Therefore $\tilde{M}\cup M''$ is the required matching of size $n-1$ in $H$.

Secondly assume that $a_1\ge a_2\ge \e n$ and $a_i<\e n$ for $i\in [3,k]$.
Since (iii) does not hold, by applying Lemma~\ref{lem:abs2}, there exists a family of disjoint absorbing $t k$-sets $\F'$ of size $|\F'|\le \sqrt\a n$ such that each $F\in \F'$ spans a matching of size $t$ and for every crossing $k$-set $S$ of $H$, the number of $S$-perfect-absorbing sets in $\F'$ is at least $\a n$.
Let $H':=H\setminus V(M')$ and $n':=|V(H')\cap V_i|\ge (1-t\sqrt\a)n$ and $a'_i := \delta_{[k]\setminus \{i\}}(H')$.
Note that $\sum_{i=1}^k a'_i \ge \sum_{i=1}^k a_i - k t \sqrt\a n \ge (1 - 2\r/5) n'$ as $\a \ll \e, 1/t$.
As before, we may assume that $H'$ is not $(4\r/5)$-S-extremal. 
By applying Lemma \ref{lem:alm_mat} to $H'$ with parameters $2\r/5$, $\a$ and $\eta$, we obtain a matching $M''$ in $H'$ that covers all but at most $\a n$ vertices in each vertex class.
Let $U$ be the set of leftover vertices.
Since any crossing $k$-subset $S$ of $U$ has at least $\a n$ $S$-perfect-absorbing $t k$-sets in $\F'$, we can greedily absorb all the leftover vertices into $\F$. Denote by $\tilde{M}$ the resulting matching that covers $V(\F')\cup U$. We obtain a perfect matching $\tilde{M}\cup M''$ of $H$.
\end{proof}

\section{Proof of Theorem \ref{thm:ext}}

We prove Theorem~\ref{thm:ext} in this section.
Following the approach in \cite{Han14_mat}, we use the following weaker version of a result by Pikhurko \cite{Pik}.
Let $H$ be a $k$-partite $k$-graph with parts $V_1, \dots, V_k$.
Given $L\subseteq [k]$, recall that
\[
\delta_{L}(H)=\min \left\{ \deg(S): S\in \prod_{i\in L}V_i \right\}.
\]

\begin{lemma}\cite[Theorem 3]{Pik}\label{lem:Pik}
Given $k\ge 2$ and $L \subseteq [k]$, let $m$ be sufficiently large. 
Let $H$ be a $k$-partite $k$-graph with parts $V_1, \dots, V_k$ of size $m$. If
\[
\delta_{L}(H) m^{|L|} + \delta_{[k]\setminus L}(H) m^{k-|L|} \ge  \frac 32m^k,
\]
then $H$ contains a perfect matching.
\end{lemma}

%

\begin{proof}[Proof of Theorem \ref{thm:ext}]
Let $\a=\sqrt{\r}$.
Assume that $H$ is not $3\e$-D-extremal. Our goal is to find a matching in $H$ of size at least $\min\{n-1, \sum_{i=1}^k a_i\}$. Assume that $H$ is $\r$-S-extremal, namely, $\sum_{i=1}^k a_i\le n+ \r n$ and there is an independent set $C \subseteq V(H)$ such that  $|C \cap V_i| \ge n - a_i - \r n$ for each $i\in [k]$.

We may assume that $\sum_{i=1}^k a_i\ge n-k+3$, as otherwise we are done by Fact~\ref{thm:mat_num}.
So we have
\begin{equation}\label{eq:sum}
n -k + 3 \le \sum_{i=1}^k a_i \le n+\r n.
\end{equation}
For each $i\in [k]$, let $C_i :=C\cap V_i$. We know that $|C_i| \ge n - a_i - \r n \ge (\e - \r) n$ from the assumption that $a_i\le (1-\e)n$.
We partition each $V_i\setminus C_i$ into $A_i\cup B_i$ such that
\begin{equation}\label{eq:A}
A_i :=\left\{x\in V_i\setminus C_i: \deg(x, C)\ge (1-\a) \prod_{j\ne i}|C_j|\right\},
\end{equation}
and $B_i :=V_i\setminus (A_i\cup C_i)$.
Moreover, let $A :=\bigcup_{1\le i\le k} A_i$ and $B :=\bigcup_{1\le i\le k} B_i$.

\begin{claim}\label{clm:size}
For $i\in [k]$, we have
\begin{enumerate} [$(1)$]
  \item $a_i\le |A_i| + |B_i|\le a_i + \r n$,
  \item $|B_i|\le \a n$, and
  \item $a_i-\a n \le |A_i| \le a_i + \r n$.
\end{enumerate}
\end{claim}

\begin{proof}
For $i\in [k]$, since $|C_i|\ge n - a_i - \r n$, we have $|A_i| + |B_i|\le a_i + \r n$. 
For any crossing $(k-1)$-set $S\subset C\setminus V_i$, we have $N(S)\subseteq A_i\cup B_i$.
By the codegree condition, we have $|A_i| + |B_i|\ge a_i$.

Let $E_i$ denote the set of the edges that consist of a $(k-1)$-set in $\prod_{j\ne i} C_i$ and one vertex in $A_i\cup B_i$. By the definition of $A_i$, we have
\[
a_i \prod_{j\ne i}|C_j| \le |E_i|\le |B_i| (1-\a)\prod_{j\ne i}|C_j| + |A_i|\prod_{j\ne i}|C_j|,
\]
which implies $a_i \le |A_i|+|B_i|-\a |B_i|$. It follows that $\a |B_i| \le |A_i|+|B_i| - a_i \le \r n$ by Part (1). Since $\a=\sqrt{\r}$, it follows that $|B_i|\le \a n$.

Part (3) follows from Parts (1) and (2) immediately. 
\end{proof}

Our procedure towards the desired matching consists of three steps. First, we remove a matching that covers all the vertices of $B$. Secondly, we remove another matching in order to have $|C''_i| - \sum_{j\ne i} |A''_j|\le \max\{1, n-\sum_{i=1}^k a_i\}$ for all $i\in [k]$, where $C''_i$ and $A''_i$ denote the set of the remaining vertices in $C_i$ and $A_i$, respectively.
Finally, we apply Lemma~\ref{lem:Pik} to get a matching that covers all but at most $\max\{1, n-\sum_{i=1}^k a_i\}$ vertices in each $V_i$.

\medskip
\noindent \emph{Step 1. Cover the vertices of $B$.}
\smallskip

For $i\in [k]$, define $t_i := \max\{0, a_i - |A_i|\}$.
By Claim~\ref{clm:size} (1), we have $|B_i|\ge a_i - |A_i|$.
Together with the definition of $t_i$ and Claim \ref{clm:size} (2), we have 
\begin{equation}\label{mat:one}
t_i\le |B_i|\le \a n \quad \text{and} \quad t_i+|A_i|\ge a_i.
\end{equation}

First we build a matching $M_1^i$ of size $t_i$ for each $i\in [k]$ and let $M_1$ be the union of them. If $t_i=0$, then $M_1^i=\emptyset$. Otherwise, since $a_i=\delta_{[k]\setminus \{i\}}(H)$ and $C$ is independent, every $(k-1)$-set in $\prod_{j\ne i} C_j$ has at least $a_i - |A_i|= t_i$ neighbors in $B_i$. We greedily pick $t_i$ disjoint edges each of which consists of a $(k-1)$-set in $\prod_{j\ne i} C_j$ and one vertex in $B_i$.

Next for each $i$, we greedily build a matching $M_2^i$ that covers all the remaining vertices in $B_i$ and let $M_2$ be the union of them.
Indeed, for each of the remaining vertices $v\in B_i$ with $i\neq 1$, we pick one uncovered $(k-2)$-set $S'$ in $\prod_{j\ne i, 1} C_j$, and one uncovered vertex in $N(\{v\} \cup S') \subseteq V_{1}$. For each of the remaining vertices in $v\in B_1$, we pick one uncovered $(k-2)$-set $S'$ in $\prod_{j\ne 1, 2} C_j$, and one uncovered vertex in $N(\{v\} \cup S') \subseteq V_{2}$.
Since the number of vertices in $V_i$ covered by the existing matchings is at most $|M_1\cup M_2|\le |B| \le k\a n < \e n \le a_2\le a_1$, we can always find an uncovered vertex from $N(\{v\} \cup S')$.

For $i\in [k]$, let
\[
A_i':=A_i\setminus V(M_1\cup M_2), \,C_i':=C_i\setminus V(M_1\cup M_2) \text{ and } V_i':=V_i\setminus V(M_1\cup M_2).
\]

\medskip
\noindent \emph{Step 2. Adjust the sizes of $A'_i$ and $C'_i$.}
\smallskip

In this step, we will build a small matching $M_3$ in order to adjust the sizes of  $A'_i$ and $C'_i$.

\begin{claim}\label{claim:mat3}
There exists a matching $M_3$ of size at most $2k\r n$ in $H[\bigcup_{i=1}^kV_i']$ so that $|C_i'\setminus V(M_3)|-\sum_{j\ne i}|A_j'\setminus V(M_3)| = r$ for some integer $0\le r \le \max\{1, n-\sum_{i=1}^k a_i\}$.
\end{claim}

\begin{proof}
Let $n':=|V_i'|=|A_i'|+|C_i'|$.
Let $s_0:=|C_i'|-\sum_{j\ne i}|A_j'|=n'-\sum_{j=1}^k|A_j'|$, which is independent of $i$.
We claim that $-2k\r n \le s_0 \le n-\sum_{i=1}^k a_i$. Indeed,
\begin{align*}
s_0 &\ge (n-\left( |M_1| + |M_2| \right))- |A|\ge n - |B| - |A| \overset{\text{Claim}~\ref{clm:size}}{\ge} n - \sum_{i=1}^k (a_i + \r n) \overset{\eqref{eq:sum}}{\ge} - 2k \r n.
\end{align*}
On the other hand, since $V(M_1)\cap A=\emptyset$ and $|M_1| = \sum_{j=1}^k t_j$, we have
\begin{align*}
s_0 &\le n - \left( |M_1| + |M_2| \right) - \left(\sum_{j=1}^k|A_j| - |M_2| \right) = n-\sum_{j=1}^k(t_j+|A_j|)\overset{\eqref{mat:one}}{\le} n-\sum_{j=1}^k a_j.
\end{align*}

If $s_0\ge 0$, then set $M_3=\emptyset$ and we are done. 
Otherwise, we build $M_3$ by adding edges that contain two or three vertices of $A$ one by one until $s\in \{0, 1\}$, where $s:= (n'-|M_3|)-\sum |A_j'\setminus V(M_3)|$.
This will be done in the next few paragraphs.
Note that since $s_0 \ge -2k\r n$ and adding an edge to $M_3$ increases $s$ by one or two, we will have $|M_3|\le 2k\r n$.

Now we show how to build $M_3$.
First assume that $a_3 \ge 2k\a n$.
In this case we greedily choose the edges of $M_3$ until $s\in \{0, 1\}$ by picking two uncovered vertices, one from $A'_2$ and one from $A'_3$, an uncovered $(k-3)$-set in $\prod_{j\in [4, k]}C'_j$, and one uncovered vertex in $V_1'$ by the degree condition.
To see why we can find these edges, first, we can always pick two uncovered vertices in $A'_2\cup A_3'$ because by Claim~\ref{clm:size} (3),
\begin{equation}\label{eq:Ai'}
    |A_i'|\ge |A_i| - |M_2|\ge a_i - \a n - k\a n \ge 2k\r n
\end{equation}
for $i=2,3$.
Secondly, we can find an uncovered $(k-3)$-set in $\prod_{j\in [4, k]}C'_j$ because
\begin{equation}\label{eq:Ci'}
    |C'_j| \ge |C_j| - |M_1\cup M_2|\ge \e n - \r n - k\a n\ge 2k\r n.
\end{equation}
Thirdly, we can find the desired vertex in $V_1$ because the number of covered vertices in $V_1$ is at most $|B| + 2k \r n \le 2k\a n< a_1$.

Next assume that $|A_1|\ge (1/2+\e)n$.
In this case we greedily choose the edges of $M_3$ until $s\in \{0, 1\}$ by picking an uncovered vertex in $A_2'$, an uncovered $(k-2)$-set in $\prod_{j\in [3, k]}C'_j$, and by the degree condition, one uncovered vertex in $A_1'$.
To see why we can find these edges, first, we can pick an uncovered $(k-1)$-set $S\in A_2'\times \prod_{i\in [3,k]}C_i'$ because of \eqref{eq:Ai'} and \eqref{eq:Ci'}.
Secondly, note that $a_1\ge |A_1| - \r n\ge (1/2+\e-\r)n$ and
\[
|A_1'|\ge |A_1| - |M_2| \ge (\tfrac12+\e)n - k\a n = (\tfrac12+\e - k\a)n.
\]
Thus, $S$ has at least $a_1 - (n - |A_1'|) \ge 2k\r n$ neighbors in $A'_1$ so we can find an uncovered neighbor of $S$.

Now we assume that $|A_1| < (1/2+\e)n$ and $a_3 < 2k\a n \le \e n$.
In this case we show that (ii) holds, i.e., $H$ is $3\e$-D-extremal.
First, $a_1\le |A_1| + \a n < (1/2+\e+\a)n$.
Since $a_i\le a_3$ for $i\in [3,k]$ and $ \sum_{i=1}^k a_i\ge n-k+3$, we have
\[
a_2 \ge \sum_{i=1}^k a_i - a_1 - (k-2) a_3 \ge n - k+3 - (\tfrac{1}2+\e+\a)n - 2k(k-2) \a n \ge (\tfrac{1}2-2\e)n,
\]
i.e., $(1/2-2\e)n\le a_2\le a_1\le (1/2+2\e)n$.
By Claim~\ref{clm:size} (3), $|A_2|\le (1/2+2\e)n+\r n\le (1/2+3\e)n$ and $|A_i|\le a_i + \r n \le 3k\a n$ for $i\in [3,k]$.
The lower bounds on $a_1, a_2$ implies that $|A_1|, |A_2|\ge (1/2-2\e)n - \a n \ge (1/2-3\e)n$.
Finally, let $x$ be the number of crossing $k$-sets in $V(H)$ that intersect $A_i$ for some $i\in [3,k]$, then $x\le (k-2)3k\a n \cdot n^{k-1} \le 3k^2\a n^k$.
Let $y_1$ be the number of non-edges in $H[A_1, B_2\cup C_2, \dots, B_k\cup C_k]$ and let $y_2$ be the number of non-edges in $H[B_1\cup C_1, A_2, B_3\cup C_3,\dots, B_k\cup C_k]$.
By the definition of $A$ and $|B_i|\le \a n$, for $1\le i\le k$ we have
\[
y_i \le |A_i| \cdot \a \prod_{j\neq i}|C_j| + \sum_{j\neq i}|B_j|\cdot n^{k-1} \le \a n^k + (k-1) \a n^k = k\a n^k.
\]
Note that $|E_{odd}(A, B\cup C)\setminus E(H)|\le x+ y_1 + y_2\le 3k^2\a n^k + 2\cdot k\a n^k \le \e n^k$.
So (ii) holds, a contradiction.
\end{proof}

\medskip
\noindent \emph{Step 3. Cover the remaining vertices.}
\smallskip

Let $M_3$ and $r$ be as in Claim~\ref{claim:mat3}.
For each $i\in [k]$, let
\[
A_i'':=A_i'\setminus V(M_3), \, C_i'':=C_i'\setminus V(M_3) \text{ and } V_i'':=V_i'\setminus V(M_3).
\]
By the definitions of $M_1, M_2, M_3$, we have $|M_1\cup M_2\cup M_3|\le k\a n + 2k\r n \le (k+1)\a n$.
By Claim \ref{clm:size} (3), for each $i\in [k]$, we have
\begin{align*}
|A_i''|\ge |A_i|-|M_1\cup M_2\cup M_3|\ge (a_i-\a n) - (k+1) \a n \ge a_i - 2k\a n.
\end{align*}
Recall that $a_1\ge a_2\ge \e n$, by $\r \ll \e$, we have
\begin{align}\label{eq:A3}
|A_1''|, |A_2''| \ge a_2 - 2k\a n \ge \e n/2.
\end{align}
By Claim \ref{claim:mat3}, we have
\begin{align}\label{eq:s}
0\le r = |C_i''|-\sum_{j\ne i}|A_j''|\le \max\left\{1, n-\sum_{i=1}^k a_i \right\}.
\end{align}
For $i\in [k]$, since $a_i\le (1-\e)n$, we get that $|C_i|\ge \e n - \r n \ge 2(k+1)\a n$.
Thus,
\begin{align}\label{eq:C3}
|C_i''|\ge |C_i| - |M_1\cup M_2\cup M_3| \ge |C_i| - (k+1)\a n \ge |C_i|/2.
\end{align}

Now we greedily cover the vertices of $A_3'', \dots, A_k''$ with disjoint edges of $H$. 
Indeed, for every $3\le i\le k$ and every vertex $v\in A_i''$, we find a neighbor of $v$ from $\bigcup_{j\ne i} C''_i$. 
By \eqref{eq:A} and \eqref{eq:C3}, at most
\[
\a \prod_{j\neq i} |C_j| \le 2^{k-1}\a \prod_{j\neq i} |C_j''|
\]
crossing $(k-1)$-sets in $\prod_{j\neq i} C_j''$ are \emph{not} neighbors of $v$.
Since $|C_i''| = \sum_{j\ne i}|A_j''|+r$, at least $\min \{|A_1''|+r, |A_2''|+r\} \ge \e n/2$ vertices of $C_i''$ remain at the end of the greedy process.
The greedily algorithm works because $(\e n/2)^{k-1} \ge 2^{k-1}\a n^{k-1} > 2^{k-1}\a \prod_{j\neq i} |C_j''|$. 

Let $M_4$ be the resulting matching in this step.
Let $m_i:=|A_i''|$ for all $i=1,2$.
Note that there are $m_2+r$ and $m_1+r$ remaining vertices in $C_1''$, $C_2''$, respectively, and $m_1+m_2+r$ remaining vertices in $C_i''$ for $i\ge 3$.
Our goal is to cover all the remaining vertices of $H$. 
For $i=1,2$, let $C_1^2$ be a set of $m_2$ vertices in $C_1''\setminus V(M_4)$ and let $C_2^1$ be a set of $m_1$ vertices in $C_2''\setminus V(M_4)$; for $i\in [3,k]$, we can partition all but $r$ vertices of $C_i''\setminus V(M_4)$ into $C_i^1$ of size $m_1$ and $C_i^2$ of size $m_2$.
Therefore, we get $k$-partite $k$-graphs $H_i := H[A_i''\cup \bigcup_{\ell\ne i} C_{\ell}^i]$ for $i=1, 2$. Below we verify the assumptions of Lemma~\ref{lem:Pik} for $H_i$.

First, for $i\in [2]$ and any $(k-1)$-set $S\in \prod_{\ell\ne i} C_\ell^i$, the number of its non-neighbors in $A_i\cup B_i$ is at most
\[
|A_i| + |B_i| - a_i \overset{\text{Claim}~\ref{clm:size}}{\le}  \r n \overset{\eqref{eq:A3}}{\le} \r \cdot \frac{2}{\e} m_i \le \a m_i,
\]
as $\r \ll \e$ and $\a = \sqrt\r$.
So we have
\[
\delta_{[k]\setminus\{i\}}(H_i)\ge m_i - \a m_i= (1-\a)m_i.
\]
Next, for any $v\in A_i''$, by \eqref{eq:A} the number of its non-neighbors in $\prod_{\ell\ne i} C_\ell^i$ is at most
\[
\a \prod_{\ell \ne i}|C_\ell| < \a n^{k-1} \overset{\eqref{eq:A3}}{\le} \a \left(\frac{2}{\e} m_i \right)^{k-1} \le \sqrt\a m_i^{k-1},
\]
which implies that $\delta_{\{i\}}(H_i) \ge (1-\sqrt\a) m_i^{k-1}$. Thus, we have
\[
\delta_{\{i\}}(H_i) m_i + \delta_{[k]\setminus\{i\}}(H_i) m_i^{k-1} \ge (1- \sqrt\a) m_i^{k-1} m_i + (1 - \a)m_i m_i^{k-1} > \frac 32 m_i^k, 
\]
since $\r$ is small enough.
By Lemma~\ref{lem:Pik}, we find a perfect matching $M_5^i$ in $H_i$ for each $i\in [2]$.
Let $M_5 :=M_5^1\cup M_5^2$, then $M_1\cup M_2\cup M_3\cup M_4\cup M_5$ is a matching in $H$ of size at least $n - r$.
If $r\le 1$, then we obtain a matching of size at least $n-1$.
Otherwise, since $0<r\le n - \sum_{i=1}^k a_i$, we get a matching of size at least $\sum_{i=1}^k a_i$.
\end{proof}

\section{Proof of Theorem~\ref{thm:ext2}}

We call a binary vector $\bfv\in \{0,1\}^k$ \emph{even} (otherwise \emph{odd}) if it contains an even number of coordinates that have value $1$.
Let $EV_k$ denote the set of all even vectors in $\{0,1\}^k$. Note that $|EV_k|= 2^{k-1}$.
Let $H=(V,E)$ be a $k$-partite $k$-graph with parts $V_1, \dots, V_k$.
Suppose $V$ also has a partition $A\cup B$, and let $A_i:=A\cap V_i$ and $B_i:=B\cap V_i$ for $i\in [k]$.
Recall that a set $S\subseteq V$ is \emph{even} (or \emph{odd}) if $|S\cap A|$ is even (or odd) and 
$E_{even}(A, B)$ consists all \emph{crossing} even $k$-subsets of $V$. 
Given a vector $\bfv\in \{0,1\}^k$, we write $V^\bfv = V^{\bfv}_1\cup \cdots \cup V^{\bfv}_k$, where $V^{\bfv}_i :=A_i$ if $\bfv|_{V_i}=1$ and $V^{\bfv}_i :=B_i$ otherwise.
Let $H(\bfv):=H[V^{\bfv}]$.
For any crossing $k$-set $S\in V^{\bfv}$, we say that $\bfv$ is the \emph{location vector} of $S$.
For $v\in V$, we define $\overline{\deg}_H(v):=\prod_{j\neq i}|V_j| - \deg_H(v)$, which is the degree of $v$ in the complement of $H$ under the same $k$-partition. Let
$\overline{\delta_1}(H):= \max_{v\in V(H)} \overline{\deg}_H(v)$.

The following theorem is a key step in the proof of Theorem~\ref{thm:ext2}. 
\begin{theorem}\label{thm:even_ext}
Suppose $1/n \ll \eta \ll\e_0, 1/k$ and $n$ is an even integer.
Let $H$ be a $k$-partite $k$-graph with parts $V_1, \dots, V_k$ of size $n$.  Suppose $A\cup B$ is a partition of $V(H)$ with $A_i:=A\cap V_i$ and $B_i:=B\cap V_i$ such that
\begin{enumerate}[$(i)$]
\item $|A_1|= |A_2|= n/2$,
\item $|A_i|=0$ or $\e_0 n\le |A_i|\le (1-\e_0) n$ for $i\ge 3$,
\item for any even vector $\bfv$, $\overline{\delta_1}(H(\bfv))\le \eta n^{k-1}$.
\end{enumerate}
Then $H$ contains a matching of size $n-1$.
Furthermore, if $|A|$ is even, then $H$ contains a perfect matching.
\end{theorem}

To prove Theorem~\ref{thm:even_ext}, we need the following simple result.

\begin{lemma}\label{lem:even_mat} 
Given a set $V$ of $kn$ vertices for some even integer $n$, let $V_1\cup \cdots \cup V_k$ and  $A\cup B$ be two partitions of $V$ such that $|V_1|=\cdots =|V_k|=n$ and $|A_1|= |A_2|= n/2$, where $A_i:=A\cap V_i$ and $B_i:=B\cap V_i$. 
Let $H= (V, E_{even}(A, B))$.
If $|A|$ is odd, then $H$ contains a matching of size $n-1$; if $|A|$ is even, then $H$ contains a perfect matching.
\end{lemma}

\begin{proof}
%

We first prove the case when $|A|$ is even by induction on $n$. 
The base case $n=2$ is simple: we divide the $2k-2$ vertices in $\bigcup_{i\ge 2} V_i$ arbitrarily into two $(k-1)$-sets and add the vertices of $V_1$ to make both sets even (these two $k$-sets have the same parity because $|A|$ is even). 
For the induction step, assume $n\ge 4$ (as $n$ is even). By picking two vertices in $V_i$, $i\ge 3$, with the same parity, we find two disjoint crossing $(k-2)$-sets in $\bigcup_{i\ge 3} V_i$ with the same parity.
We next extend them to two even $k$-sets by adding four vertices, one from each of $A_1, A_2, B_1, B_2$.
Since both $k$-sets are even, after deleting them, we can apply the inductive hypothesis. 

For the case when $|A|$ is odd, we apply the previous case after moving one vertex from $\bigcup_{i\in [3,k]}A_i$ to $B$ (note that $\sum_{i\in [3,k]}|A_i|$ is odd because $|A_1|+|A_2|=n$ is even).
Since exactly one edge has the `wrong' parity, we obtain a matching of $H$ of size $n-1$.
\end{proof}

We also use the following result of Daykin and H\"aggkvist~\cite{DaHa} while proving Theorem~\ref{thm:even_ext}.

\begin{theorem}\cite{DaHa}\label{thm:DaHa}
Let $1/n\ll 1/k$.
If $H$ is a $k$-partite $k$-graph with parts of size $n$ such that $\delta_1(H)\ge (1-1/k) (n^{k-1} -1)$,
then $H$ contains a perfect matching.
\end{theorem}

\medskip
\begin{proof}[Proof of Theorem~\ref{thm:even_ext}]
We first note that for any $\bfv\in EV_k$ and for arbitrary subsets $U_i \subseteq V^\bfv_i$, $i\in [k]$, such that $|U_i|\ge \eta^{1/(2k)} n$, $(iii)$ implies that
\begin{equation}\label{eq:deg}
\overline{\delta_1}(H[U_1,\dots, U_k]) \le \eta n^{k-1} \le \eta^{1/2} \prod_{2\le i\le k} |U_i|.
\end{equation}

We now apply Lemma~\ref{lem:even_mat} to $H' := (V, E_{even}(A, B))$ and conclude that $H'$ contains a matching $M$ of size at least $n-1$; moreover, $M$ is perfect if $|A|$ is even.
Let $S:=V\setminus V(M)$.
For each $\bfv\in EV_k$, let $m_{\bfv}$ be the number of edges in $M$ with location vector $\bfv$.
Then $\sum_{\bfv\in EV_k} m_{\bfv}=|M|$ as all the edges in $M$ are even.

It suffices to build a matching of $H$ that consists of $m_{\bfv}$ edges with location vector $\bfv$ for each $\bfv\in EV_k$. 
Let us partition $EV_k$ into $\V_1 \cup \V_2$ such that $\V_1$ consists of all $\bfv$ with $m_\bfv < \eta^{1/(2k)} n$.
For each $\bfv\in \V_1$, we greedily find a matching of size $m_{\bfv}$ in $H_\bfv$.
To see why this is possible, note that in total at most $2^{k-1}\eta^{1/(2k)} n \le \e_0^{2} n$ edges of $M$ have their location vectors in $\V_1$. 
Consequently the number of the crossing $(k-1)$-sets in $V^\bfv_2 \cup \dots \cup V^\bfv_k$ that intersect these edges is at most
\[
\e_0^2 n \sum_{2\le i\le k} \prod_{2\le j\le k, j\ne i} |V^\bfv_j| \le (k-1)\e_0 \prod_{2\le i\le k} |V^\bfv_i|
\]
because $|V_i^\bfv|\ge \e_0 n$ for $2\le i\le k$.
By~\eqref{eq:deg}, for any $v\in V_1^\bfv$, we have $\deg_{H(\bfv)}(v) \ge (1-\eta^{1/2}) \prod_{2\le i\le k} |V^\bfv_i| > (k-1)\e_0 \prod_{2\le i\le k} |V^\bfv_i|$ -- this guarantees the existence of the desired matchings for all $\bfv \in \V_1$.

Next we arbitrarily divide the remaining vertices of $V\setminus S$ into balanced vertex partitions $U_\bfv = U^\bfv_1 \cup \dots \cup U^\bfv_k$, $\bfv\in \V_2$, such that $U^{\bfv}_i \subseteq V^\bfv_i$ and $|U^{\bfv}_1|= \cdots = |U^\bfv_k|= m_{\bfv}$ -- this is possible because $\sum_{\bfv\in EV_k} m_{\bfv}=|M|$.
By~\eqref{eq:deg}, we know that $\delta_1(H[U_{\bfv}]) \ge (1-\eta^{1/2})m_{\bfv}^{k-1}\ge (1-1/k) m_{\bfv}^{k-1}$ as $\eta$ is small enough.
We thus apply Theorem~\ref{thm:DaHa} to each $H[U_{\bfv}]$ and get a perfect matching of $H[U_{\bfv}]$.
Putting all the matchings that we obtained together gives a matching of size $|M|$ in $H$.
\end{proof}

%

\begin{proof}[Proof of Theorem~\ref{thm:ext2}]
Pick a new constant $\e_0$ such that $\e \ll \e_0 \ll 1/k$.
We assume that $\sum_{i=1}^k a_i\ge n-k+3$ -- otherwise we are done by Fact~\ref{thm:mat_num}.
Moreover, suppose $H$ has a vertex bipartition $A\cup B=V_1\cup \cdots \cup V_k$ such that
\begin{itemize}
\item[($\dagger$)] $(1/2-\e)n\le a_1, a_2, |A_1|, |A_2| \le (1/2+\e)n$, and $a_i\le \e n$, $3\le i\le k$, 
\item[($\ddagger$)] $|E_{even}(A, B)\setminus E|\le \e n^k$.
\end{itemize}
Note that we obtain ($\ddagger$) after switching $A_1$ and $B_1$ if $|E_{odd}(A, B)\setminus E|\le \e n^k$. 
Furthermore, the above two properties remain valid if we switch an even number of $A_i$ with $B_i$.
Thus we may switch $A_1, A_i$ with $B_1, B_i$ whenever $|A_i| > |B_i|$ for some $i\ge 3$.
This results in $|A_i|\le |B_i|$ for all $i\ge 3$ eventually.
Moreover, by Fact~\ref{degree} (i) and ($\dagger$), we know that $\delta_1(H)\ge (1/2-\e)n^{k-1}$.

We now define \emph{atypical} vertices. Let $W$ be the set of $u\in V$ such that there exists an even $\bfv\in EV_k$ such that $u\in V^{\bfv}$ and $\overline{\deg}_{H(\bfv)}(u) > \sqrt\e n^{k-1}/2$.
Let $W_0:=W\cap (V_1\cup V_2)$.
Since each vertex in $W$ contribute at least $\sqrt\e n^{k-1}/2$ $k$-sets towards $|E_{even}(A, B)\setminus E|$ (and such a $k$-set can be counted at most $k$ times), by ($\ddagger$), we have
\begin{equation}\label{eq:W}
|W_0|\le |W| \le \frac{k\e n^k}{\sqrt\e n^{k-1}/2} \le 2k \sqrt\e n.
\end{equation}
When forming a matching of $H$, we prefer to use the edges that intersect both $A_1, A_2$ or both $B_1, B_2$ -- we will call them \emph{horizontal edges}. Correspondingly, the edges that intersect both $A_1, B_2$ or both $A_2, B_1$ are \emph{diagonal}.
We distinguish the vertices of $W_0$ that lie in fewer horizontal edges from the rest of $W_0$.
For $i=1,2$, let $W_{A_i}$ be the set of vertices of $W_0\cap A_i$ that lie in less than $\e_0 n^{k-1}$ horizontal edges; similarly let $W_{B_i}$ be the set of vertices of $W_0\cap B_i$ that lie in less than $\e_0 n^{k-1}$ horizontal edges.

Define $B_i^0:=(B_i\setminus W_{B_i})\cup W_{A_i}$ for $i=1,2$ and $B_i^0:=B_i$ for $i\ge 3$. Let $A_i^0:=V_i\setminus B_i^0$ for all $i$.
Let $A^0:=\bigcup_{i\in [k]} A_i^0$ and $B^0 := V\setminus A^0$.
Finally, let
\[
q := |B_2^0| - |B_1^0| = |B_2| - |B_1| + |W_{A_2}| + |W_{B_1}| - |W_{A_1}| - |W_{B_2}|.
\]
By ($\dagger$) and \eqref{eq:W}, $|q|\le 2\e n + 2k \sqrt\e n\le 3k\sqrt\e n$.
By relabelling $V_1$ and $V_2$ if necessary, we may assume that $q\ge 0$. Note that we still have $|A_i| \le |B_i|$ for all $i\ge 3$.

Our goal is to remove a small matching and possibly some crossing $k$-sets (non-edges) from $H$ such that we can apply Theorem~\ref{thm:even_ext} to the remaining subgraph of $H$.
To achieve the goal, we conduct the following five steps: we remove disjoint matchings $M_1, \dots, M_4$ in the first four steps and a balanced vertex set $S_5$ in the fifth step.
For $1\le j\le 4$, we define $A^{j}:=A^{j-1}\setminus V(M_j)$, $B^{j}:=B^{j-1}\setminus V(M_j)$, and $V^j:=A^j\cup B^j$.
Let $A^{5}:=A^{4}\setminus S_5$, $B^{5}:=B^{4}\setminus S_5$ and $V^5:=A^5\cup B^5$.
For $1\le j\le 5$ and $1\le i\le k$, define $A_i^{j}:=A^{j}\cap V_i$, $B_i^{j}:=B^{j}\cap V_i$, and $V_i^j:=A_i^j\cup B_i^j$.

\medskip
\noindent\textit{Step 1. Reducing the gap between $|B_1^0|$ and $|B_2^0|$.}
Our first matching $M_1$ is crucial for balancing the sizes of $B_1^0$ and $B_2^0$, and this is the only place that we need the exact codegree condition.
Let $H_1:=H[A_1^0\cup B_2^0 \cup V_3 \cup \dots \cup V_k]$ and note that
\[
\delta_{[k]\setminus \{1\}}(H_1) \ge a_1 - |B_1^0|, \quad \delta_{[k]\setminus \{2\}}(H_1) \ge a_2 - (n - |B_2^0|)
\]
and $\delta_{[k]\setminus \{i\}}(H_1)\ge a_i$ for $3\le i\le k$.
So we have
\begin{align*}
\sum_{i=1}^k\delta_{[k]\setminus \{i\}}(H_1) \ge \sum_{i=1}^k a_i + |B_2^0| - |B_1^0| - n = q - n + \sum_{i=1}^k a_i.
\end{align*}
By $(\dagger$), we have $\sum_{i=1}^k a_i\le n + k\e n$.
Thus,
\begin{equation*}
q - n + \sum_{i=1}^k a_i \le q + k\e n\le 4k\sqrt\e n < \min_{i\in [k]} |V(H_1)\cap V_i|-k.
\end{equation*}
If $q - n + \sum_{i=1}^k a_i>0$, then Fact~\ref{thm:mat_num} provides a matching $M'$ of size $q - n + \sum_{i=1}^k a_i$ in $H_1$.
Let $M_1:=M'$ if $\sum_{i=1}^k a_i\le n$ and let $M_1\subseteq M'$ be a (sub)matching of size $q$ if $\sum_{i=1}^k a_i > n$.
Otherwise let $M_1:=\emptyset$.
So we have $|M_1|\le q\le 3k\sqrt\e n$ in all cases.

\medskip
\noindent \textit{Step 2. Cleaning $V_1$ and $V_2$.}
In this step we find a matching $M_2$ that covers all the remaining vertices of $W_0$ and uses the same amount of the vertices from $A_1^0$ and $A_2^0$. 
Let $W'_0 := (W_{A_1}\cup W_{A_2}\cup W_{B_1}\cup W_{B_2}) \setminus V(M_1) $ and $W''_0 := W_0\setminus (W'_0\cup V(M_1))$. 
We cover the vertices of $W''_0$ and $W'_0$ as follows.
\begin{enumerate}
\item By definition, each vertex $u\in W''_0$ lies in at least $\e_0 n^{k-1}$ horizontal edges, i.e., those that intersect both $A_1$ and $A_2$, or intersect both $B_1$ and $B_2$.
By~\eqref{eq:W}, among these edges, at least $\e_0 n^{k-1}/2$ horizontal edges do not intersect $W$, so they intersect both $A_1^0$ and $A_2^0$, or intersect both $B_1^0$ and $B_2^0$.
We greedily cover the vertices of $W''_0$ by these disjoint horizontal edges.
\item Since $\delta_1(H)\ge (1/2-\e)n^{k-1}$,
by definition, every vertex $u\in W'_0$ lies in at least $(1/2-\e)n^{k-1} - \e_0 n^{k-1}$ diagonal edges. 
By the definitions of $A_1^0, A_2^0, B^0_1, B^0_2$ and $\e\ll \e_0$, each $u\in W'_0$ lies in at least
$(1/2-\e)n^{k-1} - \e_0 n^{k-1} - |W_0| n^{k-2}\ge \e_0 n^{k-1}$ edges that intersect both $A_1^0$ and $A_2^0$, or both $B_1^0$ and $B_2^0$. We greedily cover the vertices of $W'_0$ by such edges.
\end{enumerate}

To see why the above process is possible, we note that when finding an edge for a vertex $u$, the number of vertices that we need to avoid is at most $|V(M_1)| + k|W_0|\le 3k^2\sqrt\e n + 2k^2 \sqrt\e n= 5k^2\sqrt\e n$ by~\eqref{eq:W} and $|M_1|\le 3k\sqrt\e n$. 
Hence these vertices lie in at most $5k^2 \sqrt\e n^{k-1} < \e_0 n^{k-1}/2$ crossing $(k-1)$-sets, so we can find an edge that covers $u$ and avoids all the existing edges.

Let us bound $|B_2^2| - |B_1^2|=|B_2^0\setminus V(M_1\cup M_2)| - |B_1^0\setminus V(M_1\cup M_2)|$.
By the definition of $M_1$, we have $|B_1^0\cap V(M_1)|=0$ and $|B_2^0\cap V(M_1)|=|M_1|$. By the definition of $M_2$, we have $|B_1^0\cap V(M_2)| = |B_2^0\cap V(M_2)|$.
Thus,
\[
|B_2^2| - |B_1^2| = |B_2^0| - |B_1^0| - |M_1| = q - |M_1|.
\]
Note that
\begin{align*}
q-|M_1|=
\begin{cases}
n - \sum_{i=1}^k a_i & \text{ if } q - n + \sum_{i=1}^k a_i\ge 0 \text{ and } \sum_{i=1}^k a_i\le n; \\
0 & \text{ if } q - n + \sum_{i=1}^k a_i\ge 0\text{ and }\sum_{i=1}^k a_i> n;\\
q\le n - \sum_{i=1}^k a_i &\text{ if } q - n + \sum_{i=1}^k a_i<0.
\end{cases}
\end{align*}
So we have
\begin{equation}\label{eq:B_diff}
0\le |B_2^2| - |B_1^2| = q - |M_1| \le \max\left\{0, n - \sum_{i=1}^k a_i \right\} \le k-3.
\end{equation}

\medskip
\noindent \textit{Step 3. Cleaning $V_3,\dots, V_k$.}
Let $X$ consist of all $x\in A_i\setminus W$ for some $i\ge 3$ such that $|A_i|\le 2\e_0 n$.
In this step we build a matching $M_3$ which covers all the remaining vertices of $W$ and the vertices of $X$ and satisfies that
\begin{equation}\label{eq:step3}
-1 \le |B^2_1\cap V(M_3)| - |B^2_2 \cap V(M_3)| \le 0.
\end{equation}

Let $W'$ be the set of vertices in $(W_3\cup \cdots \cup W_k)\setminus V(M_1\cup M_2)$ that are contained in at least $3k^2\e_0 n^{k-1}$ horizontal edges.
Let $W'':=(W_3\cup \cdots \cup W_k)\setminus (V(M_1\cup M_2)\cup W')$.
Since $\delta_1(H)\ge (1/2-\e)n^{k-1}$, by definition, each $u\in W''$ is contained in at least $(1/2-\e)n^{k-1} - 3k^2\e_0 n^{k-1}$ diagonal edges. 
Note that by ($\dagger$), we have $|B_1| |A_2|\le (1/4+3\e)n^2$. 
Then since $u$ lies in at most $|B_1| |A_2| n^{k-3}\le (1/4+3\e)n^{k-1}$ edges that intersect both $B_1$ and $A_2$, there are at least $3k^2\e_0 n^{k-1}$ edges that contain $u$ and intersect both $A_1$ and $B_2$.
Note that by symmetry, the same statement holds for $u$, $A_2$ and $B_1$.
Finally, for any vertex $x\in X$, assume that $x\in A_i$ for some $3\le i\le k$.
Since the binary vectors $\bfv\in \{0,1\}^k$ with exactly two $1$'s are even, the fact that $x\not\in W$ implies that $x$ is contained in at least
\[
\prod_{j\in [k]\setminus \{1,i\}} |B_j| \cdot |A_1| - \frac12 \sqrt\e n^{k-1} \ge \frac{n^{k-1}}{2^k} -  \frac12 \sqrt\e n^{k-1}>  3k^2\e_0 n^{k-1}
\]
edges in $A_1\cup (\bigcup_{2\le j\le k, j\neq i}B_j)\cup A_i$, and in at least $3k^2\e_0 n^{k-1}$ edges in $ B_1\cup A_2 \cup (\bigcup_{3\le j\le k, j\neq i}B_j)\cup A_i$, where we used $|A_1||B_2|\ge n^2/8$ and $|B_i|\ge n/2$ for $3\le i\le k$ in the first inequality.
\begin{enumerate}
\item We first greedily find $|W'|$ disjoint horizontal edges such that each of them contains one vertex of $W'$ and no other vertices from $W\cup X \cup V(M_1\cup M_2)$. 
\item Next, we split $W''\cup X$ arbitrarily to $W_1''$ and $W_2''$ of sizes $\lfloor |W''\cup X| /2 \rfloor$ and $\lceil |W''\cup X| /2 \rceil$, respectively. We greedily find $|W''_1|$ disjoint edges such that each of them contains one vertex $u\in W_1''$, one vertex from each of $B_1$ and $A_2$, and no other vertices from $W\cup X \cup V(M_1\cup M_2)$; moreover, if $u\in A_i\subseteq X$, then the edge is taken in $ B_1\cup A_2 \cup (\bigcup_{3\le j\le k, j\neq i}B_j)\cup A_i$.

Finally, we greedily find $|W''_2|$ disjoint edges such that each of them contains one vertex $u\in W_2''$, one vertex from each of $A_1$ and $B_2$, but no other vertices from $W\cup X \cup V(M_1\cup M_2)$; moreover, if $u\in A_i\subseteq X$, then the edge is taken in $A_1\cup (\bigcup_{2\le j\le k, j\neq i}B_j)\cup A_i$.  
\end{enumerate}
The above process is possible because when considering a vertex $u$, the number of vertices that we need to avoid is at most $|V(M_1)| + |V(M_2)|+ k|W| +k\cdot 2(k-2)\e_0 n\le |V(M_1)| + 2k|W|+k\cdot 2(k-2)\e_0 n< 3k^2\e_0 n$ because of \eqref{eq:W}, the facts $|M_1|\le 3k\sqrt\e n$ and $|X|\le 2(k-2)\e_0 n$. Hence these vertices lie in less than $3k^2\e_0 n^{k-1}$ $(k-1)$-sets, so we can always find a desired edge that covers $u$ and avoids all the existing edges.
Let $M_3$ be the matching obtained in this step.
Note that~\eqref{eq:step3} holds by construction.

\medskip
\noindent \textit{Step 4. Balancing the sizes of $B_1^3$ and $A_2^3$.}
Let $m:=|B_1^3|-|A_2^3|$.
We find a matching $M_4$ of size $|m|$ as follows.
If $m\ge 0$, then $M_4$ consists of $m$ disjoint edges from $B^3$ that are disjoint from $M_1\cup M_2\cup M_3$.
Since $(0,\dots, 0)\in EV_k$, this can be done since $H[(0,\dots,0)]$ is almost complete.
Otherwise $M_4$ consists of $|m|$ disjoint edges with location vector $(1,1,0,\dots, 0)$ that are disjoint from $M_1\cup M_2\cup M_3$ -- this is possible because $H[(1,1,0,\dots,0)]$ is almost complete.

After removing $M_4$, the resulting sets $B_1^4$ and $A_2^4$ satisfy $|B_1^4| = |A_2^4|$.
The definition of $M_4$ implies that $ |B^2_1\cap V(M_4)| = |B^2_2 \cap V(M_4)|$.
Together with~\eqref{eq:B_diff} and~\eqref{eq:step3}, this gives
\begin{align*}
-1\le |B_2^4| - |B_1^4| = |B_2^2| - |B_1^2| + |B_1^2\cap V(M_3)| - |B_2^2\cap V(M_3)| \le k-3.
\end{align*}

\medskip
\noindent \textit{Step 5. Balancing the sizes of $B_1^4$ and $B_2^4$.} Let $t:= |B_2^4|-|B_1^4|$ and thus $-1\le t\le k-3$.
If $t>0$, then $n - \sum_{i=1}^k a_i \ge t> 0$.
Let $S_5$ be a $kt$-set in $V^4$ with $t$ vertices from each of $A_1^4$, $B_2^4$, and $V_i^4$ for $3\le i\le k$ such that $|A^4\setminus S_5|$ is even.
The requirement that $|A^4\setminus S_5|$ is even can be easily fulfilled if any $A_i^4$, $i\ge 3$, is not empty. On the other hand, if all $A_3^4, \dots, A_k^4$ are empty, then since $|B_2^4\setminus S_5|=|B_1^4|$, we have $|A_1^4\setminus S_5|=|A_2^4|$ and consequently,
\[
|A^4\setminus S_5|=|A_1^4\setminus S_5|+|A_2^4| = 2 |A_2^4|
\] is even.
Since $t \le n - \sum_{i=1}^k a_i$, to complete the proof, it suffices to find a perfect matching in $V^5:= V^4\setminus S_5$.
If $t=-1$, then let $S_5$ be a $k$-set with one vertex from each of $B_1^4, A_2^4$ and $V_i^4$ for $3\le i\le k$ such that $|A^4\setminus S_5|$ is even -- this can be achieved by the same argument as in the $t>0$ case. Again it suffices to find a perfect matching in $V^5$.
At last, if $t =0$ then set $S_5=\emptyset$. In this case $|A^5| = |A^4\setminus S_5|$ may be odd; however, it suffices to find a matching of size $|V^5|-1$ in $H[V^5]$.
In summary, it remains to find a perfect matching in $H[V^5]$ if $|A^5|$ is even and a matching of size $|V^5|-1$ otherwise. 
This will follow from Theorem~\ref{thm:even_ext} after we verify its assumptions.

Let $n':=|V^5_1|$ and $H':=H[V^5]$.
Note that
\[
|M_1\cup M_2\cup M_3|\le |M_1| + |W| + |X|\le 3k\sqrt\e n + 2k \sqrt\e n + 2k\e_0 n, \quad \text{and}
\]
\[
|M_4|=||B_1^3| - |A_2^3||\le |B_1| - |A_2| + |M_1\cup M_2\cup M_3|\le 3k\e_0 n,
\]
where $|B_1| - |A_2|\le 2\e n$ by ($\dagger$).
Note that we have $V(M_4)\cap A_i=\emptyset$ for $3\le i\le k$, and when building $M_3$, we may use the vertices of $A_i$, $3\le i\le k$, of size at least $2\e_0 n$ only when we cover the vertices of $W$.
Thus for $3\le i\le k$, if $|A_i^5|\neq 0$, then
\begin{align*}
|A_i^5|&\ge |A_i| - |M_1| - |V(M_2\cup M_3)\cap A_i| - |S_5\cap A_i|\\
&\ge 2\e_0 n - 3k\sqrt\e n - 2k \sqrt\e n - (k-3) \ge \e_0 n \ge \e_0 n',
\end{align*}
and $|A_i^5|\le |A_i|\le n/2<(1-\e_0)n'$.
Moreover, by the choice of $M_4$ and $S_5$, we have $|A_2^5|=|B_1^5|=|B_2^5|$, and thus $|A_1^5|=|A_2^5|=n'/2$.
In particular, this means that $n'$ is even.
Finally, note that
\begin{align*}
n' &= n - |M_1| - |M_2| - |M_3| - |M_4| - |S_5|/k \\ 
&\ge n - 3k\sqrt\e n - 2k\sqrt\e n - 2k\e_0 n - 3k\e_0 n - (k-3)\ge (1-6k\e_0) n.
\end{align*}
So for any $\bfv\in EV_k$, and any vertex $u\in H'(\bfv)$, since $u\notin W$ and $n$ is large enough, we have
\[
\overline{\deg}_{H'(\bfv)}(u)\le \sqrt\e n^{k-1}/2 < \sqrt\e n'^{k-1}.
\]
So we are done by Theorem~\ref{thm:even_ext} with $\eta = \sqrt\e$.
\end{proof}

\section*{Acknowledgement}
We thank two referees for their careful reading and detailed comments that improve the presentation of the paper. 
In particular, we are grateful to one referee for pointing out an error in the earlier version and suggesting a simpler proof of Lemma~\ref{lem:even_mat} .

\bibliographystyle{plain}
\bibliography{Apr2015}

\begin{thebibliography}{10}

\bibitem{AGS}
R.~Aharoni, A.~Georgakopoulos, and P.~Spr{\"u}ssel.
\newblock Perfect matchings in {$r$}-partite {$r$}-graphs.
\newblock {\em European J. Combin.}, 30(1):39--42, 2009.

\bibitem{AFHRRS}
N.~Alon, P.~Frankl, H.~Huang, V.~R{\"o}dl, A.~Ruci{\'n}ski, and B.~Sudakov.
\newblock Large matchings in uniform hypergraphs and the conjecture of {E}rd{\H
  o}s and {S}amuels.
\newblock {\em J. Combin. Theory Ser. A}, 119(6):1200--1215, 2012.

\bibitem{CzKa}
A.~Czygrinow and V.~Kamat.
\newblock Tight co-degree condition for perfect matchings in 4-graphs.
\newblock {\em Electron. J. Combin.}, 19(2):Paper 20, 16, 2012.

\bibitem{DaHa}
D.~E. Daykin and R.~H{\"a}ggkvist.
\newblock Degrees giving independent edges in a hypergraph.
\newblock {\em Bull. Austral. Math. Soc.}, 23(1):103--109, 1981.

\bibitem{HPS}
H.~H\`an, Y.~Person, and M.~Schacht.
\newblock On perfect matchings in uniform hypergraphs with large minimum vertex
  degree.
\newblock {\em SIAM J. Discrete Math}, 23:732--748, 2009.

\bibitem{Han14_mat}
J.~Han.
\newblock Near perfect matchings in {$k$}-uniform hypergraphs.
\newblock {\em Combin. Probab. Comput.}, 24(5):723--732, 2015.

\bibitem{Han15_mat}
J.~Han.
\newblock Near perfect matchings in $k$-uniform hypergraphs {II}.
\newblock {\em SIAM J. Discrete Math}, 30:1453--1469, 2016.

\bibitem{Han16_mat}
J.~Han.
\newblock Perfect matchings in hypergraphs and the {E}rd{\H{o}}s matching
  conjecture.
\newblock {\em SIAM J. Discrete Math}, 30:1351--1357, 2016.

\bibitem{Han14_poly}
J.~Han.
\newblock Decision problem for perfect matchings in dense $k$-uniform
  hypergraphs.
\newblock {\em Trans. Amer. Math. Soc.}, 369(7):5197--5218, 2017.

\bibitem{HT}
J.~Han and A.~Treglown.
\newblock The complexity of perfect matchings and packings in dense
  hypergraphs.
\newblock {\em arXiv:1609.06147}.

\bibitem{KM1}
P.~Keevash and R.~Mycroft.
\newblock A geometric theory for hypergraph matching.
\newblock {\em Mem. Amer. Math. Soc.}, 233(1098):vi+95, 2015.

\bibitem{Khan1}
I.~Khan.
\newblock Perfect matchings in 3-uniform hypergraphs with large vertex degree.
\newblock {\em SIAM J. Discrete Math.}, 27(2):1021--1039, 2013.

\bibitem{Khan2}
I.~Khan.
\newblock Perfect matchings in 4-uniform hypergraphs.
\newblock {\em J. Combin. Theory Ser. B}, 116:333--366, 2016.

\bibitem{KO06mat}
D.~K{\"u}hn and D.~Osthus.
\newblock Matchings in hypergraphs of large minimum degree.
\newblock {\em J. Graph Theory}, 51(4):269--280, 2006.

\bibitem{KOT}
D.~K{\"u}hn, D.~Osthus, and A.~Treglown.
\newblock Matchings in 3-uniform hypergraphs.
\newblock {\em J. Combin. Theory Ser. B}, 103(2):291--305, 2013.

\bibitem{LM1}
A.~Lo and K.~Markstr\"om.
\newblock F-factors in hypergraphs via absorption.
\newblock {\em Graphs Combin.}, 31(3):679--712, 2015.

\bibitem{LWY}
H.~Lu, Y.~Wang, and X.~Yu.
\newblock Almost perfect matchings in $k$-partite $k$-graphs.
\newblock {\em SIAM J. Discrete Math.}, 32:522--533, 2018.

\bibitem{MaRu}
K.~Markstr\"{o}m and A.~Ruci\'{n}ski.
\newblock Perfect {M}atchings (and {H}amilton {C}ycles) in {H}ypergraphs with
  {L}arge {D}egrees.
\newblock {\em European J. Comb.}, 32(5):677--687, July 2011.

\bibitem{Pik}
O.~Pikhurko.
\newblock Perfect matchings and {$K^3_4$}-tilings in hypergraphs of large
  codegree.
\newblock {\em Graphs Combin.}, 24(4):391--404, 2008.

\bibitem{RR}
V.~R\"odl and A.~Ruci\'nski.
\newblock Dirac-type questions for hypergraphs — a survey (or more problems
  for {E}ndre to solve).
\newblock {\em An Irregular Mind}, Bolyai Soc. Math. Studies 21:561--590, 2010.

\bibitem{RRS06}
V.~R\"odl, A.~Ruci\'nski, and E.~Szemer\'edi.
\newblock A {D}irac-type theorem for 3-uniform hypergraphs.
\newblock {\em Combin. Probab. Comput.}, 15(1-2):229--251, 2006.

\bibitem{RRS06mat}
V.~R{\"o}dl, A.~Ruci{\'n}ski, and E.~Szemer{\'e}di.
\newblock Perfect matchings in uniform hypergraphs with large minimum degree.
\newblock {\em European J. Combin.}, 27(8):1333--1349, 2006.

\bibitem{RRS09}
V.~R{\"o}dl, A.~Ruci{\'n}ski, and E.~Szemer{\'e}di.
\newblock Perfect matchings in large uniform hypergraphs with large minimum
  collective degree.
\newblock {\em J. Combin. Theory Ser. A}, 116(3):613--636, 2009.

\bibitem{TrZh12}
A.~Treglown and Y.~Zhao.
\newblock Exact minimum degree thresholds for perfect matchings in uniform
  hypergraphs.
\newblock {\em J. Combin. Theory Ser. A}, 119(7):1500--1522, 2012.

\bibitem{TrZh13}
A.~Treglown and Y.~Zhao.
\newblock Exact minimum degree thresholds for perfect matchings in uniform
  hypergraphs {II}.
\newblock {\em J. Combin. Theory Ser. A}, 120(7):1463--1482, 2013.

\bibitem{TrZh15}
A.~Treglown and Y.~Zhao.
\newblock {A note on perfect matchings in uniform hypergraphs}.
\newblock {\em Electron. J. Combin.}, 23:P1.16, 2016.

\bibitem{Zang_thesis}
C.~Zang.
\newblock {\em Matchings and Tilings in hypergraphs}.
\newblock PhD thesis, Georgia State University, 2016.

\bibitem{Zang16talk}
C.~Zang.
\newblock Matchings in $k$-partite $k$-uniform hypergraphs.
\newblock AMS Spring Southeastern Sectional Meeting, University of Georgia,
  Athens, GA, 2016.

\bibitem{zsurvey}
Y.~Zhao.
\newblock Recent advances on {D}irac-type problems for hypergraphs.
\newblock In {\em Recent Trends in Combinatorics}, volume 159 of {\em the IMA
  Volumes in Mathematics and its Applications}. Springer, New York, 2016.

\end{thebibliography}

\end{document}